\documentclass[12pt]{amsart}
\usepackage[margin=1in]{geometry}
\newcommand{\R}{\mathbb{R}}

\newcommand{\N}{\mathbb{N}}

\renewcommand{\epsilon}{\varepsilon}

\usepackage[english]{babel}
\usepackage[alphabetic]{amsrefs}
\usepackage{amsmath,amssymb,amsfonts,amsthm,enumerate,hyperref,mathtools}
\usepackage{url}
\usepackage{graphicx,epstopdf,color}

\usepackage[most]{tcolorbox}

\numberwithin{equation}{section}

\newtheorem{theorem}{Theorem}[section]
\newtheorem{lemma}[theorem]{Lemma}

\newtheorem{proposition}[theorem]{Proposition}
\newtheorem{corollary}[theorem]{Corollary}

\renewcommand{\leq}{\leqslant}
\renewcommand{\le}{\leqslant}

\renewcommand{\ge}{\geqslant}

\newcommand{\Per}{\operatorname{Per}}
\newcommand{\Reg}{\operatorname{Reg}}
\renewcommand{\epsilon}{\varepsilon}
\newcommand{\e}{\varepsilon}

\allowdisplaybreaks

\title[A strict maximum principle for nonlocal minimal surfaces]{A strict maximum principle \\ for nonlocal minimal surfaces}

\thanks{SD was supported
by the Australian Future Fellowship FT230100333.
OS was supported by the NSF grant DMS-2055617.
EV was supported by the Australian Laureate Fellowship FL190100081. We would like to thank the Referees, whose valuable comments have contributed
to the improvement of this paper.}

\author[Serena Dipierro,
Ovidiu Savin, and Enrico Valdinoci]{Serena Dipierro${}^{(1)}$\and
Ovidiu Savin${}^{(2)}$
\and
Enrico Valdinoci${}^{(1)}$}

\begin{document}
\maketitle

{\scriptsize \begin{center} (1) -- Department of Mathematics and Statistics\\
University of Western Australia\\ 35 Stirling Highway, WA6009 Crawley (Australia)\\
\end{center}
\scriptsize \begin{center} (2) --
Department of Mathematics\\
Columbia University\\
2990 Broadway, NY 10027
New York (USA)
\end{center}
\bigskip

\begin{center}
E-mail addresses:
{\tt serena.dipierro@uwa.edu.au},
{\tt savin@math.columbia.edu},
{\tt enrico.valdinoci@uwa.edu.au}
\end{center}
}
\bigskip\bigskip

\begin{abstract}
In the setting of fractional minimal surfaces, we prove that if two
nonlocal minimal sets are one included in the other and share a common boundary point, then they must necessarily coincide.

This strict maximum principle is not obvious, since the surfaces may touch at an irregular point, therefore a suitable blow-up analysis must be combined with a bespoke regularity theory to obtain this result.

For the classical case, an analogous result was proved by Leon Simon.
Our proof also relies on a Harnack Inequality for nonlocal minimal surfaces that has been recently introduced by Xavier Cabr\'e and Matteo Cozzi and which can be seen as a fractional counterpart of a classical result by Enrico Bombieri and Enrico Giusti.

In our setting, an additional difficulty comes from the analysis of the corresponding nonlocal integral equation on a hypersurface, which presents a remainder whose sign and fine properties need to be carefully addressed.
\end{abstract}

\section{Introduction}

The maximum principle is a fundamental tool to study geometric properties of
minimal hypersurfaces. Roughly speaking, it says that if two sets
with smooth boundaries are included one into the other and touch at some point, then the mean curvature of the inner set at the contact point is larger than that of the outer set.

The strict maximum principle aims at imposing a separation between surfaces with the same mean curvature.
While this property is rather straightforward for smooth hypersurfaces, it becomes much more delicate in the presence of singularities. For classical minimal surfaces, such a strong maximum principle has been established by Leon Simon in the celebrated article~\cite{MR906394}, also relying on a Harnack Inequality for minimal surfaces which was put forth by Enrico Bombieri and Enrico Giusti in~\cite{MR308945}.

Related results were put forth in~\cite{MR482508}. See also~\cite{MR1017330}
for a strict maximum principle for stationary
integer rectifiable varifolds in which
one of them is smooth and~\cite{MR1402732} for stationary hypersurfaces under suitable assumptions
and relying on an ``extrinsic'' Harnack inequality.
\medskip

The goal of this article is to establish a strict maximum principle in the framework of nonlocal minimal surfaces.
This result can be seen as a fractional counterpart of the main result in~\cite{MR906394}
and also leverages a recent Harnack Inequality for nonlocal minimal surfaces introduced by Xavier Cabr\'e and Matteo Cozzi in~\cite{MR3934589}. \medskip

The setting that we consider is that of nonlocal minimal surfaces, as introduced in~\cite{MR2675483}. Namely, given a (bounded and Lipschitz) domain~$\Omega\subset\R^n$
and a (measurable) set~$E\subseteq\R^n$, we say that~$E$ is $s$-minimal in~$\Omega$ if, whenever
a (measurable) set~$F$ coincides with~$E$ outside~$\Omega$, it holds that
$$ \Per_s(E,\Omega)\le\Per_s(F,\Omega),$$
where
\begin{eqnarray*} \Per_s(E,\Omega)&:=&\iint_{(E\cap\Omega)\times(E^c\cap\Omega)}\frac{dx\,dy}{|x-y|^{n+s}}\\
&&\qquad+
\iint_{(E\cap\Omega)\times(E^c\cap\Omega^c)}\frac{dx\,dy}{|x-y|^{n+s}}+
\iint_{(E\cap\Omega^c)\times(E^c\cap\Omega)}\frac{dx\,dy}{|x-y|^{n+s}}.\end{eqnarray*}
Here above\footnote{We observe that the kernel notation is in agreement with the parameter choice in~\cite{MR2675483}.
Indeed, the kernel in~\cite[page~1112]{MR2675483} was defined as~$|x-y|^{-n-2\sigma}$ with~$\sigma\in\left(0,\frac12\right)$ and
the parameter~$s:=2\sigma\in(0,1)$ was then introduced in~\cite[Section~2]{MR2675483}.} and in what follows, $s\in(0,1)$ is a fractional parameter
and the superscript ``c'' denotes the complement set in~$\R^n$.\medskip

Nonlocal minimal surfaces are a fascinating, and very difficult topic of investigation.
While their regularity is known in the plane (see~\cite{MR3090533})
and up to dimension~$7$ when~$s$ is close to~$1$ (see~\cite{MR3107529}),
the full regularity theory of these objects is not well-understood.
Though no examples of singular sets are presently known,
we believe that nonlocal minimal surfaces do develop singularities, therefore, for the validity of
a strict maximum principle, in general it does not suffice to take into consideration only regular points.\medskip

In this setting, our main result is the following:

\begin{theorem}\label{MAIN}
Let~$r>0$ and $p\in\R^n$.
Let~$E_1$, $E_2$ be~$s$-minimal sets in~$B_r(p)$, with~$E_1\subseteq E_2$.
Assume that~$p\in(\partial E_1)\cap(\partial E_2)$. 

Then, $E_1=E_2$.
\end{theorem}

We recall that a better understanding of nonlocal minimal surfaces is important
also in connection with hybrid heat equations (see~\cite{MR2564467}),
geometric motions (see~\cite{MR2487027, MR3401008, MR3713894, MR3778164, MR3951024, MR4104832, MR4175821} ),
nonlocal soap bubble problems (see~\cite{MR3836150, MR3881478}),
capillarity problems (see~\cite{MR3717439, MR4404780}), long-range phase coexistence models (see~\cite{MR2948285, MR4581189}), etc.

Interestingly, the nonlocal perimeter functional interpolates the classical perimeter as~$s\nearrow1$ (see~\cite{MR1945278, MR1942130, MR2782803, MR2765717}) with a suitably weighted Lebesgue measure as~$s\searrow0$ (see~\cite{MR1940355, MR3007726}). In this spirit, on the long run, it is expected that
nonlocal minimal surfaces can serve as auxiliary tools to better understand classical minimal surfaces as well (see~\cite{MR4116635, HARDY, case})
and can provide new perspectives on classical geometric objects by bringing forth tools
alternative to, and
different than, differential geometry (see~\cite{INSPIRED}).\medskip

Other classes of kernels have been considered in the literature, including general singular kernels and
smooth kernels, see e.g.~\cite{MR3981295, MR3930619}. It is of course interesting to investigate to which extent one can generalize the methods presented here to other cases as well.
\medskip

It would be also interesting to understand if a strong maximum principle holds true for stationary points of the nonlocal perimeter too. In the classical case,
results of this type have been obtained in~\cite{MR3268871}
where it is proved that the strict maximum principle holds for stationary codimension~$1$ integral~$(n-1)$-varifolds, when the singular set of at least one of them has zero~$(n-2)$-dimensional
Hausdorff measure.

Furthermore, in this context, it would be worth understanding whether the nonlocal setting could provide even stronger results than the classical one,
valid for two stationary points of the $s$-perimeter with no further assumptions, in particular
with no assumption on the Hausdorff dimension of the singular sets (the standard counterexamples
provided in the classical setting by~\cite{MR3268871} are made up by joining halfplanes at a common line and do not seem to work in the nonlocal framework).
\medskip

The rest of this paper focuses on the proof of Theorem~\ref{MAIN}. Roughly speaking is that,
by a blow-up and dimensional reduction argument, one reduces to the case in which the two $s$-minimal sets under consideration touch at a point which exhibit the same limit cone
for both surfaces.

One then writes a geometric equation coming from the vanishing
of the nonlocal mean curvature of one of these surfaces and a suitable translation of the other one and then scales to have a normalized picture in which the separation between these two sheets is of order one (assuming, arguing by contradiction, that they do not coincide to start with).

Some caveat is in place, since one has to switch from the notion of solution in the smooth or viscosity sense to the one in the weak sense: for this, some energy and capacity estimates will be required.

In this scenario, after a limit procedure, one aims at applying a suitable Harnack Inequality
to obtain a contradiction between a linear separation close to the origin and the assumption that the two surfaces share the same tangent cone.
As usual in geometric
measure theory (see e.g.~\cite{MR906394}),
here and in what follows by the ``same tangent cone''
we mean that any blow-up sequence
possesses a subsequence along which the two dilated surfaces
share the same tangent cone in the limit.
That is, the two surfaces
have the same tangent cone in the ``strong sense'', namely {\em{along each convergent subsequence}} and not only along one particular convergent subsequence.

However, in the above procedure, a number of difficulties emerge, since the equations under consideration
present some complicated kernels, are not defined everywhere, and produce a remainder which needs to be specifically analyzed (in particular, we will need to check that this remainder is bounded and has a sign, to be able to infer uniform H\"older estimates
on the normalized oscillation, pass them to the limit, obtain a limit inequality,
and apply to it an appropriate version of the Harnack Inequality).\medskip

Interestingly, the regularity theory utilized in this process is twofold: first we use
H\"older estimates for a nonlocal equation which is not symmetric
to perform a compactness argument on a sequence of normalized solutions which encode
the distance between the original surfaces and thus obtain a positive viscosity
supersolution on the limit cone for a fractional Jacobi operator, then we apply to this setting the
geometric Harnack Inequality to obtain the desired conclusion.

A slightly more detailed and technical sketch of the proof will be presented in the forthcoming Section~\ref{PLANNP}: after this, we focus on the rigorous arguments needed to make the actual proof work.
More specifically,
we will present in Section~\ref{MAIN:S} a suitable geometric equation which will play
a major role in our analysis. The interplay between viscosity and weak solutions
of this equation will be discussed in Section~\ref{JSOLM-fepkjvf}, through some capacity
and approximation arguments. This will allow us to obtain a viscosity counterpart of
the Harnack Inequality in~\cite[Theorem~1.7]{MR3934589}.

The geometric H\"older estimates needed to pass to the limit the rescaled configuration in which the sheets are separated by an order one will be presented in Section~\ref{LKMHOLE}
(actually, these estimates are of general flavor and can find application in other geometric problems as well).

To apply these H\"older estimates
one will also need a uniform bound on solutions of fractional equations in a geometric setting
and this argument is contained in Section~\ref{LKMHOLE-1}. In turn, in our case, this uniform bound will be a consequence
of an integral bound, which is presented in Section~\ref{INTEBO}.

The proof of Theorem~\ref{MAIN} is thus completed in Section~\ref{HAR:NA}
(the details about the good set of regular boundary points
utilized in the proof being contained in Appendix~\ref{GOODAP}).

\section{The plan of the proof}\label{PLANNP}
The proof will combine the blow-up methods introduced in~\cite{MR906394} with the fine analysis of the nonlocal setting needed in our context. 
The idea of the proof is to consider blow-up sequences of the minimal surfaces
which approach the same limit cone at the origin and
take a normal parameterization of these surfaces
away from the singular set.
We then look at the difference between these parameterizations, say~$w$ (or, more precisely, $w_k$, since
it depends on the step). The inclusion of the two sets guarantees that~$w$ has a sign, say~$w\ge0$ (and, by contradiction, we can suppose that~$w>0$ at some point).

The function~$w$ may behave very badly in the vicinity of the singular set, thus we perform our analysis in the complement of a neighborhood of the singular set (we will shrink this neighborhood at the end of the proof, relying on the regularity theory of the nonlocal minimal surfaces). Moreover, since the convergence of~$w=w_k$ is only local, we restrict our attention to a given ball~$B_R$ (we will send~$R\to+\infty$ in the end).

Roughly speaking, the gist is to choose a point~$ x_\star$ (in~$B_R$ and away from a given small neighborhood of the singular set) and a sequence of sets so that 
\begin{equation}\label{cNosVy6hgdM0ojl23fev2} 2 w_k (x_\star) \ge t^{-1} w_k (t x_\star)\end{equation}
for all~$t \in (0,1]$: it is indeed possible to fulfill this inequality (or a slight modification of it)
since the two minimal surfaces share the same tangent cone (hence their difference is ``sublinear'' at the origin, see Appendix~\ref{GOODAP} for full details
about this technical construction).
We can also normalize~$w_k$ at~$x_\star$ to be~$1$, i.e. look at the normalized function~$v_k:=\frac{w_k}{w_k(x_\star)}$.

In this way, we see that~$w=w_k$ (and therefore~$v=v_k$) satisfies a suitable
nonlocal equation of geometric type (in view of Theorem~\ref{PRIMOLE-TH})
and we show that~$v$ remains locally bounded (owing to Lemma~\ref{VOL})
and therefore H\"older continuous in compact sets that avoid the singular set
(thanks to Lemma~\ref{DERGED-Ge}).

We can thereby pick a convergent subsequence to a limit function~$v_\infty$ which is a positive supersolution to a suitable Jacobi operator on the cone in the viscosity sense. {F}rom this,
we will apply our viscosity version of the weak Harnack Inequality of~\cite{MR3934589} (as stated in Proposition~\ref{LANSCDPOJFIBFerngvny83iurjf})
and deduce that~$v_\infty$ is locally bounded below by a universal constant and we contradict in this way the inequality in~\eqref{cNosVy6hgdM0ojl23fev2} (normalized and passed to the limit).

The details to carry over this plan are very technical and several complications arise
from the nonstandard form of the integro-differential equations involved (such as the domain of integration not being flat and the kernels involved not being symmetric) and by the fact of dealing with non-standard domains of definitions (such as large balls with neighbourhood of singular sets removed). These important technicalities occupy the rest of this paper.

\section{A geometric equation}\label{MAIN:S}

Now we consider the case in which two $s$-minimal sets, one included into the other, can be written locally
as a graph in terms of another smooth hypersurface (concretely, in our case, the regular part of their common limit cone, though the arguments presented are of general flavor).

In this scenario, we obtain a geometric inequality for the difference between the parameterizations of the nonlocal minimal surfaces, which is reminiscent of the Jacobi field for nonlocal minimal surfaces (see e.g.~\cite{MR3798717, MR3951024}). The advantage of this approach in general
is that these Jacobi fields leverage additional cancellations with respect to the two individual equations for the nonlocal mean curvature of the two surfaces.
The specialty of our setting is however that the two nonlocal minimal surfaces are nice graphs with respect to the limit cone only away from the possible singularities of the cone (hence, in our application,
we will have to consider graphs converging to the limit cone away from its singularities and carefully estimate the error terms).

Actually, the methods that we develop work in a greater generality, which is also useful to appreciate the geometric structures arising in the limit construction.
Thus, we split our calculations in different regions of the space
(namely, near a regular point in Lemma~\ref{PRIMOLE},
in an intermediate ring in Lemma~\ref{PRIMOLE2}, and far away in Lemma~\ref{PRIMOLE3}) and we collect these calculations
into the general form given by Theorem~\ref{PRIMOLE-TH}.

This type of results are influenced by the classical theory of minimal surfaces but present significant differences with respect to the classical case. For instance, in the local setting,
one can apply directly the theory of Jacobi fields and obtain a linear equation describing the normal displacement of a surface approaching a limit cone (see equation~(7) in~\cite{MR906394}).
Instead, in the nonlocal setting, remote interactions highly complicate the structure of the equations and a fine theory of cancellations is needed to obtain the convergence of the desired approximation and effective estimates on the remainder terms.

In the forthcoming calculations,
we use the notation \begin{equation*} \widetilde\chi_E(x):=\chi_{\R^n\setminus E}(x)-\chi_{E}(x)=\begin{cases}
1 &{\mbox{ if }}x\in\R^n\setminus E,\\
-1&{\mbox{ if }}x\in E.
\end{cases}\end{equation*}

\begin{center}
\begin{figure}[h]\tcbox{
\includegraphics[width=8.7cm]{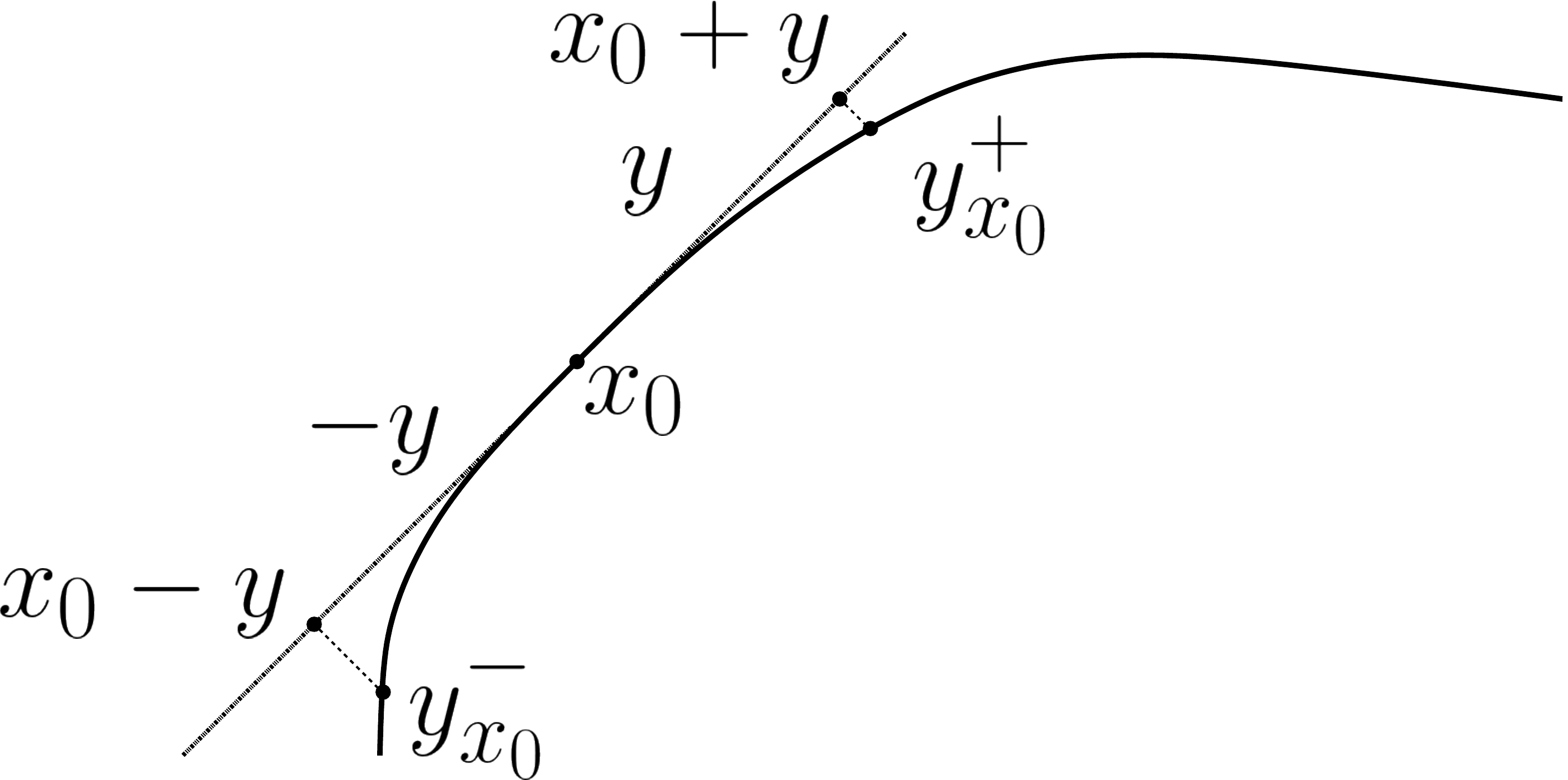}}
\caption{\footnotesize{\it Projections along the tangent hyperplane.}}\label{pkdf:A}
\end{figure}\end{center}

Also, here below we will take into consideration suitable kernels defined on a hypersurface, specifically,
a certain regular subset of~$\partial E_1$, with~$E_1\subseteq\R^n$: in this situation, given~$x_0$ on this hypersurface and
a vector~$y$ (with small norm) in the linear tangent hyperplane at~$x_0$ we will denote by~$y^+_{x_0}$
the point on the hypersurface whose projection onto the affine tangent hyperplane at~$x_0$ coincides with~$x_0+y$
and by~$y^-_{x_0}$
the point on the hypersurface whose projection onto the affine tangent hyperplane at~$x_0$ coincides with~$x_0-y$, see Figure~\ref{pkdf:A}.

Moreover, in what follows, given a point~$X$ in a small ball~$B_\delta(\overline X)$, we will consider the normal coordinates~$X = x + t\nu(x)$, 
with~$ x\in\partial E_1$ and~$t\in\R$. In principle, $x$ does not necessarily belong
to~$B_\delta(\overline X)$ (since it could falls slightly outside of it). Therefore, it is convenient to use the notation~${\mathcal{B}}_\delta(\overline X)$
to denote an open set of~$\R^n$, possibly slightly larger than~$B_\delta(\overline X)$, which contains all the points~$x$ above.

\begin{lemma}\label{PRIMOLE}
Let~$E_1\subseteq E_2\subseteq\R^n$.

Suppose that there exist~$\overline{X}\in\R^n$ and~$\delta>0$ such that~$(\partial E_1)\cap B_\delta(\overline X)$
is a hypersurface of class~$C^{2}$.
Let~$\nu$ be the unit external normal to~$E_1$ in~$B_\delta(\overline X)$.

Suppose that
$$ (E_2\setminus E_1)\cap B_\delta(\overline X)=\Big\{
x+t\nu(x),\;{\mbox{with }} x\in\partial E_1,\; t\in\big[0,w(x)\big)
\Big\}\cap B_\delta(\overline X),$$
for some function~$w$ of class~$C^{2}$ with values in~$[0,\delta/4]$.

Given~$x_0\in (\partial E_1)\cap B_{\delta/8}(\overline X)$, let
$$ \widetilde E_1:=E_1+w(x_0)\nu(x_0)
\qquad{\mbox{and}}\qquad X_0:=x_0+w(x_0)\nu(x_0).$$

Then, if~$\delta$ and~$\|w\|_{C^2({{\mathcal{B}}_\delta}(\overline X))}$ are small enough
with respect to structural quantities and the local~$C^2$-norm of~$\partial E_1$,
\begin{equation}\label{ICLN6tao2o3ir4}
\begin{split}& \frac12
\int_{B_{\delta/4}(\overline X)}\frac{\widetilde\chi_{E_2}(X)-\widetilde\chi_{\widetilde E_1}(X)}{|X-X_0|^{n+s}}\,dX\\&\qquad=\int_{(\partial E_1)\cap {{\mathcal{B}}_{\delta/4}}(\overline X)}
\big( w(x_0)-w(x)\big)\,K_1(x,x_0)
\,d{\mathcal{H}}^{n-1}_x\\&\qquad\qquad
-w(x_0)\int_{(\partial E_1)\cap {{\mathcal{B}}_{\delta/4}}(\overline X)}
\big(1-\nu(x_0)\cdot\nu(x)\big)\,K_2(x,x_0)
\,d{\mathcal{H}}^{n-1}_x+O(w^2(x_0)),\end{split}\end{equation}
where~$K_j:({(\partial E_1)\cap {{\mathcal{B}}_{\delta/4}}(\overline X)})^2\to[0,+\infty]$, with~$j\in\{1,2\}$, satisfies, for a suitable~$C\ge1$, that
\begin{equation}\label{PKSLM-CANDLEDLWD}\begin{split}&
{\mbox{for all~$x\ne x_0$, we have that }}\;
\frac{1}{C|x-x_0|^{n+s}} \leq K_j(x,x_0)\le \frac{C}{|x-x_0|^{n+s}},\\ &{\mbox{and, for all~$y\in B_\delta\setminus\{0\}$, }}
K_j(y^+_{x_0},x_0)-K_j(y^-_{x_0},x_0)=O(|y|^{1-n-s})
.\end{split}\end{equation}

The ``big~$O$'' terms and the constant~$C$ here above depend only
on the curvatures of~$\partial E_1$ in~$B_{9\delta/10}(\overline X)$, on~$\|w\|_{C^2({{\mathcal{B}}_\delta}(\overline X))}$, on~$n$ and~$s$.

Also, the kernel~$K_j$ approaches pointwise, up to a normalizing constant, the kernel~$|x-x_0|^{-n-s}$ when~$\|w\|_{C^2({{\mathcal{B}}_\delta}(\overline X))}$ tends to zero.
\end{lemma}

We point out that, consistently with the definition of~$y^\pm_{x_0}$ given above, the point~$y$ in~\eqref{PKSLM-CANDLEDLWD}
is implicitly assumed to lie in the linear tangent hyperplane at~$x_0$.

\begin{proof}[Proof of Lemma~\ref{PRIMOLE}] Given~$X \in B_\delta(\overline X)$, we write~$X=x+t\nu(x)$, with~$x\in\partial E_1$ and~$t\in\R$.
If~$\delta>0$ is small enough, the map linking~$X$ to~$(x,t)$ is a diffeomorphism and the intersection with~$\partial E_1$ occurs only for~$t=0$.

Thus, we denote the corresponding geometric Jacobian determinant by~$ J(x,t)$ (set to zero when~$x+t\nu(x)\not\in B_{\delta/4}(\overline X)$)
and we have that
\begin{equation}\label{CAL-01}\begin{split}&
\int_{B_{\delta/4}(\overline X)}\frac{\widetilde\chi_{E_2}(X)-\widetilde\chi_{\widetilde E_1}(X)}{|X-X_0|^{n+s}}\,dX\\&\qquad=
\iint_{((\partial E_1)\cap {{\mathcal{B}}_{\delta/4}}(\overline X))\times\R}
\frac{\widetilde\chi_{E_2}(x+t\nu(x))-\widetilde\chi_{\widetilde E_1}(x+t\nu(x))}{|x-x_0+t\nu(x)-w(x_0)\nu(x_0)|^{n+s}}\,J(x,t)\,d{\mathcal{H}}^{n-1}_x\,dt.\end{split}
\end{equation}

Now, given~$x\in(\partial E_1)\cap {{\mathcal{B}}_{\delta/4}}(\overline X)$, we have that
$$ \widetilde\chi_{E_2}(x+t\nu(x))=\begin{cases} 1 & {\mbox{ if }}t\ge w(x),\\-1& {\mbox{ if }}t< w(x).
\end{cases}$$
Also, $\widetilde\chi_{\widetilde E_1}(x+t\nu(x))=-1$ if and only if~$x+t\nu(x)\in \widetilde E_1$, and so if and only if~$x(t):=x+t\nu(x)-w(x_0)\nu(x_0)\in E_1$.

Now, the signed distance of~$x(t)$ to the tangent hyperplane of~$E_1$ at~$x$ is equal to
\begin{eqnarray*}
d(x):=(x(t)-x)\cdot\nu(x)=t-w(x_0)\nu(x_0)\cdot\nu(x).\end{eqnarray*}
Moreover, the projection of~$x(t)-x$ onto the tangent plane is
\begin{eqnarray*}&& (x(t)-x)- \big((x(t)-x)\cdot\nu(x)\big)\nu(x)
=t\nu(x)-w(x_0)\nu(x_0)-\big(t-w(x_0)\nu(x_0)\cdot\nu(x)\big)\nu(x)\\&&\qquad
=-w(x_0)\nu(x_0)+w(x_0)\nu(x_0)\cdot\nu(x)\nu(x)=w(x_0)\Big(\nu(x_0)\cdot\nu(x)\nu(x)-\nu(x_0)\Big),
\end{eqnarray*}
which has length equal to
$$ w(x_0)\Big|\nu(x_0)\cdot\nu(x)\nu(x)-\nu(x_0)\Big|=w(x_0)\sqrt{1-(\nu(x_0)\cdot\nu(x))^2}.
$$
Since~$\partial E_1$ detaches at most quadratically from its tangent hyperplane, we have that
$$ \widetilde\chi_{\widetilde E_1}(x+t\nu(x))=\begin{cases} 1 & {\mbox{ if }}t\ge \widetilde w(x),\\-1& {\mbox{ if }}t< \widetilde w(x),
\end{cases}$$
where
\begin{equation}\label{7648tuygfusjkfdsjhsyr874ywiu} \widetilde w(x):=w(x_0)\nu(x_0)\cdot\nu(x)+
O\big(w^2(x_0)|\nu(x)-\nu(x_0)|^2\big).\end{equation}
See Figure~\ref{tange2B}.

\begin{center}
\begin{figure}[h]\tcbox{
\includegraphics[width=8.7cm]{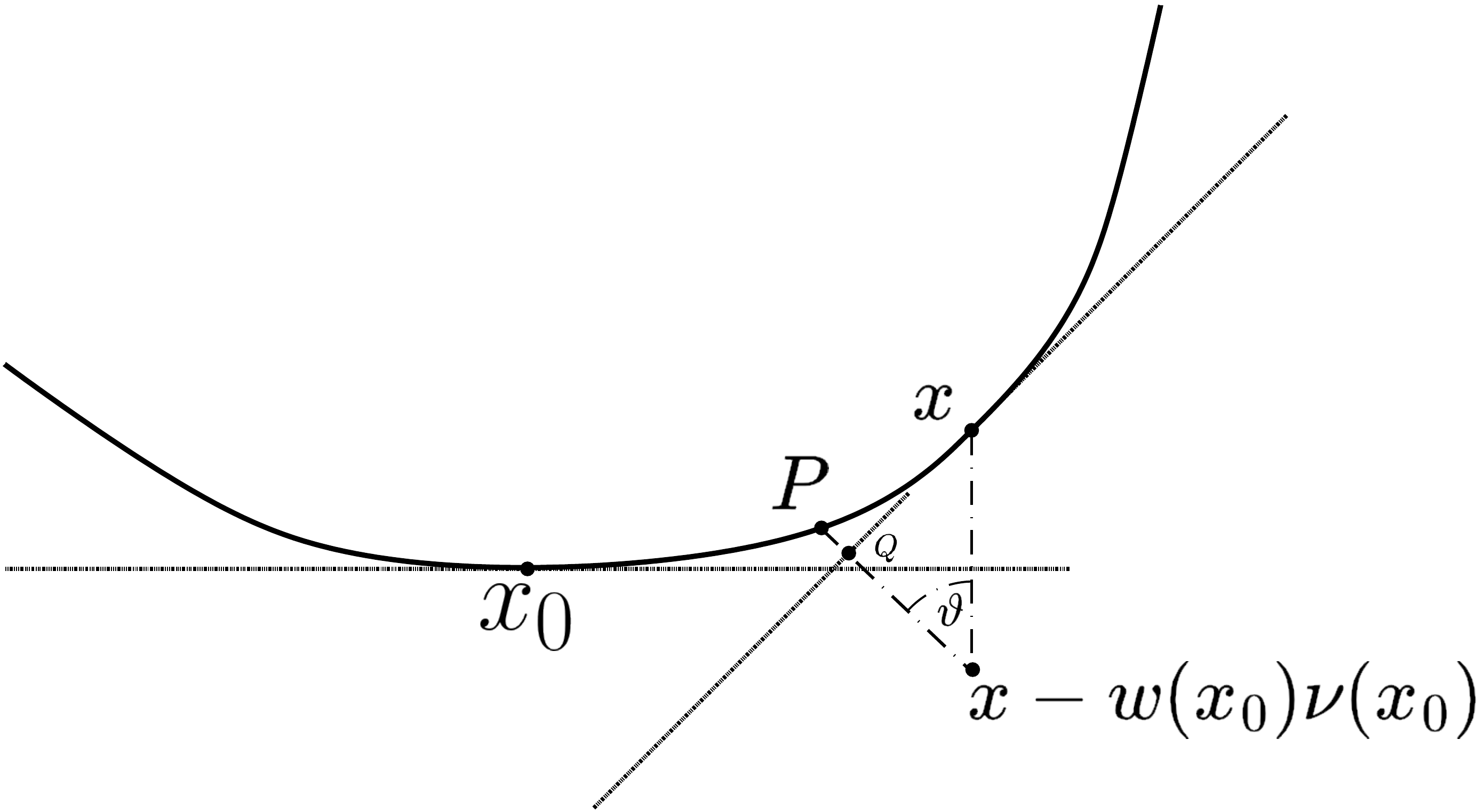}}
\caption{\footnotesize{\it Diagram to visualize~\eqref{7648tuygfusjkfdsjhsyr874ywiu}.
Here, $\vartheta$ denotes the angle between~$\nu(x_0)$ and~$\nu(x)$, the distance between~$x-w(x_0)\nu(x_0)$ and~$Q$ corresponds to~$w(x_0)\cos\vartheta$, the distance between~$x$ and~$Q$ corresponds to~$w(x_0)\sin\vartheta$ (implying that the distance between~$P$ and~$Q$ is of the order of~$w^2(x_0)\sin^2\vartheta$, with constants depending on the smoothness of~$E_1$.}}\label{tange2B}
\end{figure}\end{center}

{F}rom these observations, we infer that, for all~$x\in(\partial E_1)\cap {{\mathcal{B}}_{\delta/4}}(\overline X)$,
\begin{eqnarray*}&&
\int_\R \frac{\widetilde\chi_{E_2}(x+t\nu(x))}{|x-x_0+t\nu(x)-w(x_0)\nu(x_0)|^{n+s}}\,J(x,t)\,dt=h_1(w(x))-h_2(w(x))\\
{\mbox{and }}&&
\int_\R \frac{\widetilde\chi_{\widetilde E_1}(x+t\nu(x))}{|x-x_0+t\nu(x)-w(x_0)\nu(x_0)|^{n+s}}\,J(x,t)\,dt=h_1(\widetilde w(x))-h_2(\widetilde w(x)),
\end{eqnarray*}
where
\begin{equation}\label{Hdef-002}\begin{split}& h_1(\xi):=
\int_{\xi}^{+\infty}\frac{J(x,t)\,dt}{|x-x_0+t\nu(x)-w(x_0)\nu(x_0)|^{n+s}}\\
{\mbox{and }}\quad& h_2(\xi):=
\int_{-\infty}^{\xi} \frac{J(x,t)\,dt}{|x-x_0+t\nu(x)-w(x_0)\nu(x_0)|^{n+s}}.
\end{split}\end{equation}

We insert this information into~\eqref{CAL-01} and we find that
\begin{equation}\label{1312rfg5rdiqknwfem}
\begin{split}&
\int_{B_{\delta/4}(\overline X)}\frac{\widetilde\chi_{E_2}(X)-\widetilde\chi_{\widetilde E_1}(X)}{|X-X_0|^{n+s}}\,dX\\&\qquad=
\int_{(\partial E_1)\cap {{\mathcal{B}}_{\delta/4}}(\overline X)}
\Big( h_1(w(x))-h_1(\widetilde w(x))-h_2(w(x))
+h_2(\widetilde w(x)) \Big)\,d{\mathcal{H}}^{n-1}_x\\
&\qquad=
\int_{(\partial E_1)\cap {{\mathcal{B}}_{\delta/4}}(\overline X)}
\Big( h(w(x))-h(\widetilde w(x))\Big)\,d{\mathcal{H}}^{n-1}_x,
\end{split}
\end{equation}
where
\begin{equation}\label{Hdef-003}h:=h_1-h_2.\end{equation}

We also observe that
\begin{equation}\label{HPojasdJJ}-h'(\xi)=
-h_1'(\xi)+h_2'(\xi)=
\frac{2J(x,\xi)}{|x-x_0+\xi\nu(x)-w(x_0)\nu(x_0)|^{n+s}}.
\end{equation}
Moreover, we remark that, 
\begin{equation}\label{tryeufghds65748yhfsuyfr3t7trfgdka}
{\mbox{as~$|x-x_0|\to0$, the vector~$x-x_0$ becomes tangent to~$\partial E_1$.}}\end{equation}

Now we let
\begin{eqnarray*}&&A(x_0):={\rm Id}+w(x_0)I\!I(x_0)\\{\mbox{and }}&&
\sigma(x):=A(x_0)(x-x_0)=\big({\rm Id}+w(x_0)I\!I(x_0)\big)(x-x_0),
\end{eqnarray*}
where~$I\!I$ denotes the second fundamental form along~$\partial E_1$, extended outside the tangent space by orthogonal projection.

We observe that, for small~$w(x_0)$, 
\begin{equation}\label{quellochestoper}
|\sigma(x)|\in \left[ \frac{|x-x_0|}2,2|x-x_0|\right].\end{equation}
Here below we will take~$a$ and~$b\in\R$ with~$|a|+|b|\le C|x-x_0|$, with~$C>0$ depending on the regularity of~$w$ and~$\partial E_1$, and
we set~$\theta:=\frac{\xi-w(x_0)}{|\sigma(x)|}$ for~$\xi\in(a+w(x_0),b+w(x_0))$.
In this way, $\theta$ is bounded, thanks to~\eqref{quellochestoper}, and
we have that
\begin{equation}\label{APksjdlwRSBSQ}\begin{split}&
\frac{h(a+w(x_0))-h(b+w(x_0))}2\\=\,&-\frac12\int_{a+w(x_0)}^{b+w(x_0)} h'(\xi)\,d\xi\\=\,&\int_{a+w(x_0)}^{b+w(x_0)}
\frac{J(x,\xi)\,d\xi}{|x-x_0+\xi\nu(x)-w(x_0)\nu(x_0)|^{n+s}}\\=\,&\int_{a/|\sigma(x)|}^{b/|\sigma(x)|}
\frac{|\sigma(x)|\,J\big(x,w(x_0)+\theta|\sigma(x)|\big)\,d\theta}{\Big|x-x_0+\big(w(x_0)+\theta|\sigma(x)|\big)\nu(x)-w(x_0)\nu(x_0)\Big|^{n+s}}.
\end{split}\end{equation}

Also, recalling~\eqref{tryeufghds65748yhfsuyfr3t7trfgdka}, we have that
$$(x-x_0)\cdot\nu(x)=O(|x-x_0|^2)$$
and
\begin{equation*}
\nu(x)-\nu(x_0)=I\!I(x_0)(x-x_0)+O(|x-x_0|^2).
\end{equation*}
For this reason,
\begin{eqnarray*}&&
\big(w(x_0)+\theta|\sigma(x)|\big)\nu(x)-w(x_0)\nu(x_0)\\&=&
\big(w(x_0)+\theta|\sigma(x)|\big)\big(\nu(x_0)+I\!I(x_0)(x-x_0)+O(|x-x_0|^2)\big)
-w(x_0)\nu(x_0)\\&=&w(x_0)I\!I(x_0)(x-x_0)+
\theta|\sigma(x)|\nu(x_0)+O(|x-x_0|^2).
\end{eqnarray*}
As a result,
\begin{eqnarray*}
&&\Big|x-x_0+\big(w(x_0)+\theta|\sigma(x)|\big)\nu(x)-w(x_0)\nu(x_0)\Big|^{2}
\\&=&\Big|\big({\rm Id}+w(x_0)I\!I(x_0)\big)(x-x_0)+\theta|\sigma(x)|\nu(x_0)+O(|x-x_0|^2)\Big|^{2}\\&=&\Big|\sigma(x)+\theta|\sigma(x)| \nu(x_0)+O(|x-x_0|^2)\Big|^{2}\\
&=&(1+\theta^2)|\sigma(x)|^2+2\theta|\sigma(x)| \sigma(x)\cdot\nu(x_0)+O(|x-x_0|^3).
\end{eqnarray*}

Actually, since
\begin{eqnarray*}
\sigma(x)\cdot\nu(x_0)&=&
\big(w(x_0)I\!I(x_0)(x-x_0)\big)\cdot\nu(x_0)+O(|x-x_0|^2)\\
&=&w(x_0)\big(\nu(x)-\nu(x_0)\big)\cdot\nu(x_0)+O(|x-x_0|^2)\\&=&
w(x_0)\big(\nu(x)\cdot\nu(x_0)-1\big)+O(|x-x_0|^2)\\&=&
-\frac{w(x_0)}2\big|\nu(x)-\nu(x_0)\big|^2+O(|x-x_0|^2)\\&=&O(|x-x_0|^2),
\end{eqnarray*}
we find that
\begin{eqnarray*}
&&\Big|x-x_0+\big(w(x_0)+\theta|\sigma(x)|\big)\nu(x)-w(x_0)\nu(x_0)\Big|^{2}
=(1+\theta^2)|\sigma(x)|^2+O(|x-x_0|^3).
\end{eqnarray*}

This gives that
\begin{eqnarray*}&&
\Big|x-x_0+\big(w(x_0)+\theta|\sigma(x)|\big)\nu(x)-w(x_0)\nu(x_0)\Big|^{-(n+s)}\\&=&
\Big( (1+\theta^2)|\sigma(x)|^2+O(|x-x_0|^3)\Big)^{-\frac{n+s}2}\\&=&
(1+\theta^2)^{-\frac{n+s}2}|\sigma(x)|^{-(n+s)}\Big( 1+O(|x-x_0|)\Big)^{-\frac{n+s}2}\\&=&
(1+\theta^2)^{-\frac{n+s}2}|\sigma(x)|^{-(n+s)}\Big( 1+O(|x-x_0|)\Big)\end{eqnarray*}
and consequently, in view of~\eqref{APksjdlwRSBSQ},
\begin{equation}\label{09iuhg098uytfdP-0ipejro3fen}\begin{split}&
\frac{h(a+w(x_0))-h(b+w(x_0))}2\\&\qquad=\int_{a/|\sigma(x)|}^{b/|\sigma(x)|}
\frac{|\sigma(x)|\,J\big(x,w(x_0)+\theta|\sigma(x)|\big)\,\big(1  +O(|x-x_0|)\big)}{
(1+\theta^2)^{\frac{n+s}2} |\sigma(x)|^{n+s} }\,d\theta.
\end{split}\end{equation}

Now we take
$$ a:=w(x)-w(x_0)\qquad{\mbox{and}}\qquad b:=\widetilde w(x)-w(x_0).$$ We observe that
$$ |w(x)-w(x_0)|=O\big(\|\nabla w\|_{L^\infty(B_\delta(\overline X))}|x-x_0|\big)$$
and
\begin{eqnarray*} |\widetilde w(x)-w(x_0)|&\le& w(x_0)\big|\nu(x_0)\cdot\nu(x)-1\big|+
O\big(w^2(x_0)|x-x_0|^2\big)\\&\le& O\big(w(x_0)|x-x_0|^2\big).\end{eqnarray*}
This says that, in our setting,
\begin{equation}\label{uwoejfn39-wedf}
\frac{a}{|\sigma(x)|}=O\big( \|\nabla w\|_{L^\infty({{\mathcal{B}}_\delta}(\overline X))}\big)\qquad{\mbox{and}}\qquad
\frac{b}{|\sigma(x)|}=O\big(w(x_0)|x-x_0|\big).\end{equation}
Therefore, with reference to the integral in~\eqref{09iuhg098uytfdP-0ipejro3fen}, in this setting we have that $\theta=O\big( \|w\|_{C^1({{\mathcal{B}}_\delta}(\overline X))}\big)$ and therefore
$$ J\big(x,w(x_0)+\theta|\sigma(x)|\big)=J\big(x,w(x_0)\big)+O\big(|x-x_0|\big)=J\big(x_0,w(x_0)\big)+O\big(|x-x_0|\big),$$
with the latter quantity depending on the curvatures of~$\partial E_1$ in~$B_{9\delta/10}(\overline X)$, on $\|w\|_{C^2({{\mathcal{B}}_\delta}(\overline X))}$, on~$n$ and~$s$.

Gathering this and~\eqref{09iuhg098uytfdP-0ipejro3fen} we conclude that
\begin{equation*}\begin{split}&
\frac{h(a+w(x_0))-h(b+w(x_0))}2\\&\qquad=
\frac{|\sigma(x)|\,\big(J(x_0,w(x_0))  +O(|x-x_0|)\big)}{|\sigma(x)|^{n+s} }\,\int_{a/|\sigma(x)|}^{b/|\sigma(x)|}
\frac{d\theta}{
(1+\theta^2)^{\frac{n+s}2}}\\&\qquad=
\frac{J(x_0,w(x_0))  +O(|x-x_0|)}{|\sigma(x)|^{n+s} }\,\left(
b\Phi\left(\frac{b}{|\sigma(x)|}\right)-a\Phi\left(\frac{a}{|\sigma(x)|}\right)
\right),\end{split}\end{equation*}
where
$$ \Phi(\alpha):=\frac1{\alpha}\int_0^{\alpha}
\frac{d\theta}{(1+\theta^2)^{\frac{n+s}2}}.$$

Now we let
\begin{equation}\label{311}
{\mathcal{A}}_1(x,x_0):=
\Phi\left(\frac{w(x)-w(x_0)}{|\sigma(x)|}\right)\qquad{\mbox{and}}\qquad
{\mathcal{A}}_2(x,x_0):=
\Phi\left(\frac{\widetilde w(x)-w(x_0)}{|\sigma(x)|}\right)
\end{equation}
and we see that
\begin{eqnarray*}&&
\frac{h(w(x))-h(\widetilde w(x))}2\\&=&
\frac{J(x_0,w(x_0))  +O(|x-x_0|)}{|\sigma(x)|^{n+s} }\,\Big(
(\widetilde w(x)-w(x_0)){\mathcal{A}}_2(x,x_0)-(w(x)-w(x_0)){\mathcal{A}}_1(x,x_0)
\Big)\\&=&
\frac{J(x_0,w(x_0))  +O(|x-x_0|)}{|\sigma(x)|^{n+s} }\\ &&
\times\Big[
\Big(w(x_0) (\nu(x_0)\cdot\nu(x)-1)+O(w^2(x_0)|x-x_0|^2)\Big){\mathcal{A}}_2(x,x_0)-(w(x)-w(x_0)){\mathcal{A}}_1(x,x_0)
\Big]\\&=&
\Big(w(x_0) (\nu(x_0)\cdot\nu(x)-1)+O(w^2(x_0)|x-x_0|^2)\Big)K_2(x,x_0)-(w(x)-w(x_0))K_1(x,x_0),
\end{eqnarray*}
where $$ K_j(x,x_0):=
\frac{\big(J(x_0,w(x_0))  +O(|x-x_0|)\big)\,{\mathcal{A}}_j(x,x_0)}{|\sigma(x)|^{n+s} }
.$$
By inserting this into~\eqref{1312rfg5rdiqknwfem},
we have thereby obtained the desired result in~\eqref{ICLN6tao2o3ir4},
but we need to check~\eqref{PKSLM-CANDLEDLWD} in order to ensure the necessary cancellations to make sense of the integrals involved.

To this end, we observe that the Jacobian~$J(x_0,t)$ approaches~$1$ as~$t\to0$.
Moreover, 
$$ \Phi(\alpha)\le\frac1{\alpha}\int_0^{\alpha}
\,d\theta=1$$
and, if~$\alpha\in(-1,1)$,
$$ \Phi(\alpha)\ge \frac1{\alpha}\int_0^{\alpha}
\frac{d\theta}{2^{\frac{n+s}2}}=\frac1{2^{\frac{n+s}2}}.$$
This, \eqref{uwoejfn39-wedf} and~\eqref{311}
entail that, in the range of interest, also~${\mathcal{A}}_j(x,x_0)\in
\left[\frac1{2^{\frac{n+s}2}},1\right]$.
These considerations, together with~\eqref{quellochestoper}, establish the first claim in~\eqref{PKSLM-CANDLEDLWD}.

Additionally,
$$ |\sigma(x_0\pm y)|=|\pm A(x_0)y|=|A(x_0)y|.$$
Furthermore,
$$ \Phi(\alpha)=\frac1{\alpha}\int_0^{\alpha}
\left( 1-{\frac{n+s}2}\,\theta^2+O(\theta^4)\right)\,d\theta=1-{\frac{n+s}6}\,\alpha^2+O(\alpha^4),$$
giving that, if~$\alpha\in(-1,1)$, then~$|\Phi'(\alpha)|\le C$ and accordingly
$$ |\Phi(\alpha)-\Phi(\beta)|\le C|\alpha-\beta|.$$

We also have that
\begin{eqnarray*} &&w(y^+_{x_0})-w(y^-_{x_0})\\&&\quad=w(x_0+y)-w(x_0+y)+O(|x_0+y-y^+_{x_0}|)
+O(|x_0-y-y^-_{x_0}|) =O(|y|^2)\end{eqnarray*}
and
$$ \widetilde w(y_{x_0}^\pm)-w(x_0)=w(x_0)\big(\nu(x_0)\cdot\nu(y_{x_0}^\pm)-1\big)+O(|y|^2)=O(|y|^2).$$
As a consequence,
\begin{eqnarray*}
&&\Big|{\mathcal{A}}_1(y^+_{x_0},x_0)-{\mathcal{A}}_1(y^-_{x_0},x_0)\Big|\\&&\qquad=\left|
\Phi\left(\frac{w(y^+_{x_0})-w(x_0)}{|A(x_0)y|}\right)-\Phi\left(\frac{w(y^+_{x_0})-w(x_0)}{|A(x_0)y|}\right)\right|\\&&\qquad\leq C\left|  \frac{w(y^+_{x_0})-w(y^-_{x_0})}{|A(x_0)y|}\right|\le C|y|
\end{eqnarray*}
and
\begin{eqnarray*}
&&\Big|{\mathcal{A}}_2(y^+_{x_0},x_0)-{\mathcal{A}}_2(y^-_{x_0},x_0)\Big|\\&&\qquad=\left|
\Phi\left(\frac{\widetilde w(y^+_{x_0})-w(x_0)}{|A(x_0)y|}\right)-\Phi\left(\frac{\widetilde w(y^-_{x_0})-w(x_0)}{|A(x_0)y|}\right)\right|\le C|y|.
\end{eqnarray*}
The proof of~\eqref{PKSLM-CANDLEDLWD} is thereby complete.

The fact that~$K_1$ and~$K_2$ can be seen as perturbations of the kernel~$|x-x_0|^{-n-s}$
when
$\|w\|_{C^2({{\mathcal{B}}_\delta}(\overline X))}$ tends to zero follows since in this asymptotics~$\sigma(x)$ approaches~$|x-x_0|$, ${\mathcal{A}}_1$ and~${\mathcal{A}}_2$ approach~$\Phi(0)=1$ and the Jacobian term approaches~$J(x,0)=1$.
\end{proof}

In the next result, we will consider a set~${\mathcal{G}}$ collecting the regular points of~$\partial E_1$
(the regularity theory of nonlocal minimal surfaces allowing us to conveniently choose this set). We will also consider a shrinking~${\mathcal{G}}'$ of~${\mathcal{G}}$ \label{aojscnSMNEdcvSD3jx}
and we will implicitly assume~${\mathcal{G}}$ to be contained in a small neighborhood of~$\partial E_1$.

\begin{lemma}\label{PRIMOLE2}
Let~$\delta\in(0,1)$, $R>1+2\delta$, $\overline X\in \R^n$ and~$E_1\subseteq E_2\subseteq\R^n$.

Assume that there exists~${\mathcal{G}}\subseteq B_{2R}(\overline X)$ such that~$(\partial E_1)\cap {\mathcal{G}}$
is a hypersurface of class~$C^{2}$ 
with unit external normal~$\nu$.

Suppose that
$$ (E_2\setminus E_1)\cap {\mathcal{G}}=\Big\{
x+t\nu(x),\;{\mbox{with }} x\in\partial E_1,\; t\in\big[0,w(x)\big)
\Big\}\cap {\mathcal{G}},$$
for some function~$w$ of class~$C^{2}$ with values in~$[0,\delta/4]$.

Assume also that
\begin{equation}\label{K123SD:a}
{\mbox{the Hausdorff distance between~$\partial E_1$ and~$\partial E_2$ in~$B_{2R}(\overline X)$ is less than }}\frac{\delta}{20}.\end{equation}

Let~$x_0\in (\partial E_1)\cap B_{\delta/8}(\overline X)$, $\nu_0\in\partial B_1$,
$$ \widetilde E_1:=E_1+w(x_0)\nu_0
\qquad{\mbox{and}}\qquad X_0:=x_0+w(x_0)\nu_0.$$

Let also~${\mathcal{G}}'\Subset{\mathcal{G}}$ and assume that
\begin{equation}\label{KSMD-madu-923irj2iwrh3h4sn-22} B_{2R}(\overline X)\setminus {\mathcal{G}}'\subseteq \bigcup_{j\in\N} B_{r_j}(p_j),\end{equation}
for some balls~$\{B_{r_j}(p_j)\}_{j\in\N}$,
with
\begin{equation}\label{KSMD-madu-923irj2iwrh3h4sn-24} \eta:=\sum_{j\in\N} r_j^{n-1}<1.\end{equation}

Then, if~$\delta$ is small enough, and~$\eta$
and~$\|w\|_{L^{\infty}({{\mathcal{B}}_R}(\overline X)\cap{\mathcal{G}})}$ are small
compared to~$\delta$ and~$1/R$,
\begin{eqnarray*}&& \frac12
\int_{B_R(\overline X)\setminus B_{\delta/4}(\overline X)}\frac{\widetilde\chi_{E_2}(X)-\widetilde\chi_{\widetilde E_1}(X)}{|X-X_0|^{n+s}}\,dX\\&&\qquad=
\int_{((\partial E_1)\cap {\mathcal{G}}')\setminus {{\mathcal{B}}_{\delta/4}}(\overline X)}
\big( w(x_0)-w(x)\big)\,K_1(x,x_0)
\,d{\mathcal{H}}^{n-1}_x
\\&&\qquad\qquad\qquad-w(x_0)\int_{((\partial E_1)\cap {\mathcal{G}}')\setminus {{\mathcal{B}}_{\delta/4}}(\overline X)}
\big(1-\nu_0\cdot\nu(x)\big)\,K_2(x,x_0)
\,d{\mathcal{H}}^{n-1}_x
\\&&\qquad\qquad\qquad
+\frac12\int_{(B_R(\overline X)\setminus B_{\delta/4}(\overline X))\setminus{\mathcal{G}}'}\frac{\widetilde\chi_{E_2}(X)-\widetilde\chi_{E_1}(X)}{|X-X_0|^{n+s}}\,dX
+O\left(w(x_0)\eta\right),\end{eqnarray*}
where the ``big~$O$'' term here above depends only
on~$\delta$, $R$, $n$ and~$s$.

In addition,~$K_j:({((\partial E_1)\cap {\mathcal{G}}')\setminus {{\mathcal{B}}_{\delta/4}}(\overline X)})^2\to[0,+\infty]$, with~$j\in\{1,2\}$, satisfies~\eqref{PKSLM-CANDLEDLWD} and approaches, up to a normalizing constant, the kernel~$|x-x_0|^{-n-s}$ when~$\|w\|_{C^2({{\mathcal{B}}_R}(\overline X)\cap{\mathcal{G}})}$ tends to zero.
\end{lemma}

\begin{proof} As in the proof of Lemma~\ref{PRIMOLE}, given~$X \in (B_R(\overline X)\setminus B_{\delta/4}(\overline X))\cap{\mathcal{G}}'$, we write~$X=x+t\nu(x)$, with~$x\in\partial E_1$ and~$t\in\R$ (we recall that~${\mathcal{G}}'$ provides a small neighborhood of the regular part of~$\partial E_1$).

By~\eqref{K123SD:a}, we see that
\begin{equation*}
{\mbox{if~$|t|\ge$}}\frac\delta{10},{\mbox{ then }}
\widetilde\chi_{E_2}(x+t\nu(x))-\widetilde\chi_{\widetilde E_1}(x+t\nu(x))=0.
\end{equation*}
In this setting, if~$\delta>0$ is small enough,
\begin{equation}\label{alertenv12345565u6fc0iowfejg}
\begin{split}&
\int_{(B_R(\overline X)\setminus B_{\delta/4}(\overline X))\cap{\mathcal{G}}'}\frac{\widetilde\chi_{E_2}(X)-\widetilde\chi_{\widetilde E_1}(X)}{|X-X_0|^{n+s}}\,dX\\=\;&
\iint_{((\partial E_1)\cap ({{{\mathcal{B}}_R}(\overline X)\setminus {{\mathcal{B}}_{\delta/4}}(\overline X)})\cap{\mathcal{G}}')\times\R}
\frac{\widetilde\chi_{E_2}(x+t\nu(x))-\widetilde\chi_{\widetilde E_1}(x+t\nu(x))}{|x-x_0+t\nu(x)-w(x_0)\nu_0|^{n+s}}\,J(x,t)\,d{\mathcal{H}}^{n-1}_x\,dt,\end{split}
\end{equation}
where~$ J(x,t)$ denotes the geometric Jacobian determinant
(set to zero when~$x+t\nu(x)\not\in B_{R}(\overline X)$).

Given~$x\in(\partial E_1)\cap {\mathcal{G}}'$, we have that
\begin{equation}\label{ojdwfken2345y6y-1} \widetilde\chi_{E_2}(x+t\nu(x))=\begin{cases} 1 & {\mbox{ if }}t\ge w(x),\\-1& {\mbox{ if }}t< w(x).
\end{cases}\end{equation}
Moreover, $\widetilde\chi_{\widetilde E_1}(x+t\nu(x))=-1$ if and only if~$x+t\nu(x)\in \widetilde E_1$, and so if and only if~$x(t):=x+t\nu(x)-w(x_0)\nu_0\in E_1$.

Now, the signed distance of~$x(t)$ to the tangent hyperplane of~$E_1$ at~$x$ is equal to
\begin{eqnarray*}
d(x):=(x(t)-x)\cdot\nu(x)=t-w(x_0)\nu_0\cdot\nu(x).\end{eqnarray*}
Moreover, the projection of~$x(t)-x$ onto the tangent plane is
\begin{eqnarray*}&& (x(t)-x)- \big((x(t)-x)\cdot\nu(x)\big)\nu(x)
=t\nu(x)-w(x_0)\nu_0-\big(t-w(x_0)\nu_0\cdot\nu(x)\big)\nu(x)
\\&&\qquad\qquad\qquad=
w(x_0)\nu_0\cdot\nu(x)\nu(x)-w(x_0)\nu_0 
\end{eqnarray*}
which has length equal to
$$ \ell(x):=\sqrt{w^2(x_0)-(w(x_0)\nu_0\cdot\nu(x))^2}.
$$
We stress that~$\ell(x)\le w(x_0)$, hence if~$w(x_0)$ is small enough then~$B_{\ell(x)}(x)$ lies in~${\mathcal{G}}$ and accordingly~$\partial E_1$ detaches at most quadratically from its tangent hyperplane
in this region. This gives that
\begin{equation}\label{ojdwfken2345y6y-2} \widetilde\chi_{\widetilde E_1}(x+t\nu(x))=\begin{cases} 1 & {\mbox{ if }}t\ge \widetilde w(x),\\-1& {\mbox{ if }}t< \widetilde w(x),
\end{cases}\end{equation}
where
$$ \widetilde w(x):=w(x_0)\nu_0\cdot\nu(x)+
O\big(w^2(x_0)\big).$$

By using~\eqref{ojdwfken2345y6y-1} and~\eqref{ojdwfken2345y6y-2}, and in the notation of~\eqref{Hdef-002} and~\eqref{Hdef-003}
(with~$\nu(x_0)$ replaced by~$\nu_0$ here), we deduce that, if $x\in (\partial E_1)\cap ({{{\mathcal{B}}_R}(\overline X)\setminus {{\mathcal{B}}_{\delta/4}}(\overline X)})\cap{\mathcal{G}}'$, then
\begin{eqnarray*}
\int_{\R}
\frac{\widetilde\chi_{E_2}(x+t\nu(x))-\widetilde\chi_{\widetilde E_1}(x+t\nu(x))}{|x-x_0+t\nu(x)-w(x_0)\nu_0|^{n+s}}\,J(x,t)\,dt=h(w(x))-h(\widetilde w(x))
\end{eqnarray*}
and therefore
\begin{equation}\label{PAKSdlmext2434tee234ttndy57a7s3h0oeuwrihftg}
\begin{split}&
\iint_{((\partial E_1)\cap ({{{\mathcal{B}}_R}(\overline X)\setminus {{\mathcal{B}}_{\delta/4}}(\overline X)}\cap{\mathcal{G}}')\times\R}
\frac{\widetilde\chi_{E_2}(x+t\nu(x))-\widetilde\chi_{\widetilde E_1}(x+t\nu(x))}{|x-x_0+t\nu(x)-w(x_0)\nu_0|^{n+s}}\,J(x,t)\,d{\mathcal{H}}^{n-1}_x\,dt\\&\qquad=\int_{(\partial E_1)\cap ({{{\mathcal{B}}_R}(\overline X)\setminus {{\mathcal{B}}_{\delta/4}}(\overline X)})\cap{\mathcal{G}}'}\Big(h(w(x))-h(\widetilde w(x))\Big)\,d{\mathcal{H}}^{n-1}_x.
\end{split}
\end{equation}

Thus, we define
$$ K(x,x_0):=-\frac{h(w(x))-h(\widetilde w(x))}{2(w(x)-\widetilde w(x))}$$
and we remark that this kernel satisfies~\eqref{PKSLM-CANDLEDLWD}.
Indeed, we can take here~$K_1=K_2=K$ and, since\footnote{We observe that
if~$X=x+t\nu(x)$ with~$t\in[0,w(x))$ it follows that~$|X-x|\le w(x)\le\frac{\delta}{200}$.}
in our setting~$ |x-x_0|\ge\frac{\delta}{100}$, we only need to check the first claim in~\eqref{PKSLM-CANDLEDLWD}.

And this claim holds true, because, by~\eqref{HPojasdJJ},
\begin{eqnarray*}&&
-\frac{h(w(x))-h(\widetilde w(x))}{w(x)-\widetilde w(x)}=-
\int_0^1 h'( tw(x)+(1-t)\widetilde w(x)) \,dt\\&&\qquad=
\int_0^1
\frac{2J(x,tw(x)+(1-t)\widetilde w(x))}{|x-x_0+(tw(x)+(1-t)\widetilde w(x))\nu(x)-w(x_0)\nu_0|^{n+s}}\,dt,
\end{eqnarray*}
which, when~$  |x-x_0|\ge\frac{\delta}{100}$, is bounded from above and below by~$\frac1{|x-x_0|^{n+s}}$, up to multiplicative constants.

We also deduce from~\eqref{PAKSdlmext2434tee234ttndy57a7s3h0oeuwrihftg} that
\begin{equation}\label{APKSLdm03tugj-3orkkg}\begin{split}&
\frac12\int_{(B_R(\overline X)\setminus B_{\delta/4}(\overline X))\cap{\mathcal{G}}'}\frac{\widetilde\chi_{E_2}(X)-\widetilde\chi_{\widetilde E_1}(X)}{|X-X_0|^{n+s}}\,dX\\&\qquad=\int_{(\partial E_1)\cap ({{{\mathcal{B}}_R}(\overline X)\setminus {{\mathcal{B}}_{\delta/4}}(\overline X)})\cap{\mathcal{G}}'}
\big(\widetilde w(x)-w(x)\big)\,K(x,x_0)
\,d{\mathcal{H}}^{n-1}_x.\end{split}
\end{equation}

Now we deal with the term
\begin{eqnarray*}&&\int_{(B_R(\overline X)\setminus B_{\delta/4}(\overline X))\setminus{\mathcal{G}}'}\frac{\widetilde\chi_{E_2}(X)-\widetilde\chi_{\widetilde E_1}(X)}{|X-X_0|^{n+s}}\,dX\\ &=&
\int_{(B_R(\overline X)\setminus B_{\delta/4}(\overline X))\setminus{\mathcal{G}}'}\frac{\widetilde\chi_{E_2}(X)-\widetilde\chi_{E_1}(X)}{|X-X_0|^{n+s}}\,dX
+\int_{(B_R(\overline X)\setminus B_{\delta/4}(\overline X))\setminus{\mathcal{G}}'}\frac{\widetilde\chi_{E_1}(X)-\widetilde\chi_{\widetilde E_1}(X)}{|X-X_0|^{n+s}}\,dX.\end{eqnarray*}
Specifically, to estimate the latter term, we define~$F_1$ as the symmetric difference between~$E_1$ and~$\widetilde E_1$ and we point out that, by~\eqref{KSMD-madu-923irj2iwrh3h4sn-22},
\begin{equation}\label{09ryhfeIJSNDu8uERGIOKSJ}\begin{split}
\left|\int_{(B_R(\overline X)\setminus B_{\delta/4}(\overline X))\setminus{\mathcal{G}}'}\frac{\widetilde\chi_{E_1}(X)-\widetilde\chi_{\widetilde E_1}(X)}{|X-X_0|^{n+s}}\,dX\right|
&\le2\int_{((B_R(\overline X)\setminus B_{\delta/4}(\overline X))\setminus{\mathcal{G}}')\cap F_1}\frac{dX}{|X-X_0|^{n+s}}\\ &\le\frac{C}{\delta^{n+s}}\sum_{j\in\N}\big|B_{r_j}(p_j)\cap F_1\big|.
\end{split}\end{equation}

To complete this estimate, it is useful to observe that, for all~$r>0$, all~$\tau\in\R^n$ and all measurable sets~$L$, we have that, for some~$C>0$,
\begin{equation}\label{yufrdbfncxri7uyetuf98765}
\Big|B_{r}\cap \big((L+\tau)\setminus L\big)\Big|\le Cr^{n-1}|\tau|.
\end{equation}
Indeed, if~$|\tau|\le r$,
\begin{eqnarray*}
&&\Big|B_{r}\cap \big((L+\tau)\setminus L\big)\Big|
=\left| \int_{B_{r}\cap (L+\tau)}\,dx- \int_{B_{r}\cap L}\,dx\right|
=\left| \int_{(B_{r}-\tau)\cap L}\,dx- \int_{B_{r}\cap L}\,dx\right|\\&&\qquad=\Big|
L\cap \big((B_{r}-\tau)\setminus B_{r}\big)\Big|\le\big|(B_{r}-\tau)\setminus B_{r}\big|\le\big|B_{r+|\tau|}\setminus B_{r}\big|
\le  Cr^{n-1}|\tau|,
\end{eqnarray*}
as desired.

If instead~$|\tau|>r$,
$$ \Big|B_{r}\cap \big((L+\tau)\setminus L\big)\Big|\le |B_r|\le Cr^n\le Cr^{n-1}|\tau|.$$
These considerations establish~\eqref{yufrdbfncxri7uyetuf98765}.

Plugging~\eqref{yufrdbfncxri7uyetuf98765} into~\eqref{09ryhfeIJSNDu8uERGIOKSJ} and exploiting~\eqref{KSMD-madu-923irj2iwrh3h4sn-24}, we conclude that
$$ \left|\int_{(B_R(\overline X)\setminus B_{\delta/4}(\overline X))\setminus{\mathcal{G}}'}\frac{\widetilde\chi_{E_1}(X)-\widetilde\chi_{\widetilde E_1}(X)}{|X-X_0|^{n+s}}\,dX\right|\le\frac{Cw(x_0)}{\delta^{n+s}}\sum_{j\in\N}r_j^{n-1}\le\frac{C\eta w(x_0)}{\delta^{n+s}}.$$
The desired result follows from this inequality and~\eqref{APKSLdm03tugj-3orkkg}
(the asymptotics of the kernel being similar to those discussed in Lemma~\ref{PRIMOLE}).
\end{proof}

\begin{lemma}\label{PRIMOLE3}
Let~$\delta\in(0,1)$, $R>1+2\delta$, $\overline X\in \R^n$ and~$E_1\subseteq E_2\subseteq\R^n$.

Let~$x_0\in B_{\delta/8}(\overline X)$, $\nu_0\in\partial B_1$, $w_0\in[0,\delta/4]$
$$ \widetilde E_1:=E_1+w_0\nu_0
\qquad{\mbox{and}}\qquad X_0:=x_0+w_0\nu_0.$$

Then, 
\begin{eqnarray*}&& \frac12
\int_{\R^n\setminus B_R(\overline X)}\frac{\widetilde\chi_{E_2}(X)-\widetilde\chi_{\widetilde E_1}(X)}{|X-X_0|^{n+s}}\,dX
=\frac12
\int_{\R^n\setminus B_R(\overline X)}\frac{\widetilde\chi_{E_2}(X)-\widetilde\chi_{E_1}(X)}{|X-X_0|^{n+s}}\,dX
+O\left(\frac{w_0}{R^{1+s}}\right),\end{eqnarray*}
where the ``big~$O$'' term here above depends only
on~$n$ and~$s$.
\end{lemma}

\begin{proof} We observe that
\begin{eqnarray*}&&\left|
\int_{\R^n\setminus B_R(\overline X)}\frac{\widetilde\chi_{E_1}(X)-\widetilde\chi_{\widetilde E_1}(X)}{|X-X_0|^{n+s}}\,dX\right|\\&&\qquad=\left|
\int_{\R^n\setminus B_R(\overline X)}\frac{\widetilde\chi_{E_1}(X)}{|X-X_0|^{n+s}}\,dX-
\int_{\R^n\setminus B_R(\overline X-w_0\nu_0)}\frac{\widetilde\chi_{E_1}(X)}{|X+w_0\nu_0-X_0|^{n+s}}\,dX\right|\\&&\qquad\le\left|
\int_{\R^n\setminus B_R(\overline X)}\frac{\widetilde\chi_{E_1}(X)}{|X-X_0|^{n+s}}\,dX-
\int_{\R^n\setminus B_R(\overline X-w_0\nu_0)}\frac{\widetilde\chi_{E_1}(X)}{|X-X_0|^{n+s}}\,dX\right|+\frac{Cw_0}{R^{1+s}}\\&&\qquad\le\frac{Cw_0}{R^{1+s}},
\end{eqnarray*}
up to renaming~$C$, as usual, from line to line.
\end{proof}

Now we recall the notion of nonlocal mean curvature of a set~$E$ at a point~$X\in\partial E$, namely
$$ H^s_E(X):=\int_{\R^n} \frac{\widetilde\chi_E(Y)}{|X-Y|^{n+s}}\,dY.$$
Putting together Lemmata~\ref{PRIMOLE}, \ref{PRIMOLE2} and~\ref{PRIMOLE3},
we conclude that:

\begin{theorem}\label{PRIMOLE-TH}
Let~$E_1\subseteq E_2\subseteq\R^n$.
Let~$\delta\in(0,1)$, $R>1+2\delta$ and~$\overline X\in\R^n$.

Suppose that~$(\partial E_1)\cap {\mathcal{G}}$
is a hypersurface of class~$C^{2}$, for some~${\mathcal{G}}\subseteq\R^n$ such that~$
{\mathcal{G}}\Supset B_{2\delta}(\overline X)$,
and let~$\nu$ be the unit external normal to~$E_1$ in~${\mathcal{G}}$.

Suppose that
$$ (E_2\setminus E_1)\cap {\mathcal{G}}=\Big\{
x+t\nu(x),\;{\mbox{with }} x\in\partial E_1,\; t\in\big[0,w(x)\big)
\Big\}\cap {\mathcal{G}},$$
for some function~$w$ of class~$C^{2}$ with values in~$[0,\delta/4]$.

Assume also that
the Hausdorff distance between~$\partial E_1$ and~$\partial E_2$ in~$B_{2R}(\overline X)$ is less than~$\delta/20$.

Given~$x_0\in (\partial E_1)\cap B_{\delta/8}(\overline X)$, let
$$ \widetilde E_1:=E_1+w(x_0)\nu(x_0)
\qquad{\mbox{and}}\qquad X_0:=x_0+w(x_0)\nu(x_0).$$

Let also~${\mathcal{G}}'\Subset{\mathcal{G}}$ with~${\mathcal{G}}'\Supset B_{2\delta}(\overline X)$
and assume that
\begin{equation*} B_{2R}(\overline X)\setminus {\mathcal{G}}'\subseteq \bigcup_{j\in\N} B_{r_j}(p_j),\end{equation*}
for some balls~$\{B_{r_j}(p_j)\}_{j\in\N}$,
with
\begin{equation}\label{AKSMHAU} \eta:=\sum_{j\in\N} r_j^{n-1}<1.\end{equation}

Then, if~$\delta$ is small enough, and~$\eta$ and~$\|w\|_{C^2({{\mathcal{B}}_R}(\overline X)\cap{\mathcal{G}})}$ are small
compared to~$\delta$ and~$1/R$,
\begin{equation}\label{o9eudashiqpowudjhifeg}\begin{split}&
\frac{H^s_{E_2}(X_0)-H^s_{\widetilde E_1}(X_0)}2
=\int_{(\partial E_1)\cap {\mathcal{G}}'}
\big( w(x_0)-w(x)\big)\,K_1(x,x_0)
\,d{\mathcal{H}}^{n-1}_x\\&\qquad\qquad-w(x_0)\int_{(\partial E_1)\cap {\mathcal{G}}'}
\big(1-\nu(x_0)\cdot\nu(x)\big)\,K_2(x,x_0)
\,d{\mathcal{H}}^{n-1}_x\\&\qquad\qquad
+\frac12\int_{\R^n\setminus {\mathcal{G}}'}\frac{\widetilde\chi_{E_2}(X)-\widetilde\chi_{E_1}(X)}{|X-X_0|^{n+s}}\,dX
+O\left( w(x_0)\eta\right)
+O\left(\frac{w(x_0)}{R^{1+s}}\right),\end{split}\end{equation}
where~$K_j:((\partial E_1)\cap {\mathcal{G}}')^2\to[0,+\infty]$, with~$j\in\{1,2\}$, satisfies, for a suitable~$C\ge1$, that
\begin{equation}\label{IPSJLkear34}\begin{split}&
{\mbox{for all~$x\ne x_0$, we have that }}\;
\frac{1}{C|x-x_0|^{n+s}} \leq K_j(x,x_0)\le \frac{C}{|x-x_0|^{n+s}},\\ &{\mbox{and, for all~$y\in B_\delta\setminus\{0\}$, }}
K_j(y_{x_0}^+,x_0)-K_j(y_{x_0}^-,x_0)=O(|y|^{1-n-s}).
\end{split}\end{equation}
All the ``big~$O$'' terms here above depend only on the
curvatures of~$\partial E_1$ in~$B_{9\delta/10}(\overline X)$, on~$\|w\|_{C^2({{\mathcal{B}}_R}(\overline X)\cap{\mathcal{G}})}$, on~$\delta$, $R$, $n$ and~$s$,
except for the last one in~\eqref{o9eudashiqpowudjhifeg}, which depends only on~$n$ and~$s$. 

Also, the kernel~$K_j$
approaches, up to a normalizing constant, the kernel~$|x-x_0|^{-n-s}$ when~$\|w\|_{C^2({{\mathcal{B}}_R}(\overline X)\cap{\mathcal{G}})}$ tends to zero.
\end{theorem}

\section{A geometric Harnack Inequality for viscosity supersolutions}\label{JSOLM-fepkjvf}

Having obtained a geometric equation
for a nonnegative function in Theorem~\ref{PRIMOLE-TH}, our objective would be to apply a Harnack Inequality to it (actually, to a suitable limit of it, and this will indeed be implemented in the forthcoming Section~\ref{HAR:NA}). This is however rather delicate, since in principle, in our setting, the equation will only be valid along the regular part of a limit $s$-minimal cone, but we do not have any information about the equation along the singular points of this cone.

Moreover, the Harnack Inequality in~\cite[Theorem~1.7]{MR3934589}
is set into a smooth or distributional framework, which, in principle, does not cover the case of viscosity solutions
that will be needed here in Section~\ref{HAR:NA} (indeed, for us, the limit procedure carried out
in Section~\ref{HAR:NA} naturally casts the notion of solution into the viscosity framework).

These difficulties can be overcome by an appropriate cut-off argument, based on a fine covering of the singular points
(if any) of the limit cone, in combination with
suitable approximation arguments. 
Namely, the regularity theory of $s$-minimal cones (see~\cite{MR3090533, MR3331523})
allows one to confine these singular points within a small region. 

In our setting, the Harnack Inequality that will be applied in Section~\ref{HAR:NA} goes as follows:

\begin{proposition}\label{LANSCDPOJFIBFerngvny83iurjf}
Let~${\mathcal{R}}\subset\R^n$ and~$S:= \overline{{\mathcal{R}}}\setminus {\mathcal{R}}$. Assume that~${\mathcal{R}}$ is locally a $C^2$-smooth hypersurface and that
\begin{equation}\label{o3ePP-1-pre}
{\mbox{$S$ is closed and has Hausdorff dimension~$d\le n-3$.}}\end{equation}
Assume also that, for every~$r>0$ and~$p\in\R^n$,
\begin{equation}\label{o3ePP-1}
{\mathcal{H}}^{n-1}({\mathcal{R}}\cap B_r(p))\le Cr^{n-1},
\end{equation}
for some~$C>0$.

Let~$u:{\mathcal{R}}\to[0,+\infty)$ be H\"older continuous
and such that, for every~$x\in {\mathcal{R}}$,
\begin{equation*} \int_{\mathcal{R}} \frac{u(x)-u(y)}{|x-y|^{n+s}}\,d{\mathcal{H}}^{n-1}_y\ge0\end{equation*}
in the viscosity sense.

Then, 
$$ \inf_{{\mathcal{R}}\cap B_1} u\ge c\left( \int_{{\mathcal{R}}\cap B_1}u(x)\,d{\mathcal{H}}^{n-1}_x+\int_{{\mathcal{R}}\setminus B_1}\frac{u(x)}{|x|^{n+s}}\,d{\mathcal{H}}^{n-1}_x\right),$$
for some~$c>0$, depending only on~$n$ and~$s$.
\end{proposition}

The proof of Proposition~\ref{LANSCDPOJFIBFerngvny83iurjf} relies on~\cite[Theorem~1.7]{MR3934589},
once we reduce ourselves to the case of weak supersolutions.
Indeed, we point out that~\cite[Theorem~1.7]{MR3934589} requires the solution to be ``smooth'':
however, this smoothness is used in~\cite{MR3934589} to deal with weak supersolutions,
as stressed in~\cite{MR3934589} after~(6.20)
(also, the solution is supposed to be bounded in~\cite{MR3934589}, but one can always
reduce the situation to this case by cutting the graph of the solution at a high level, being the minimum of viscosity supersolutions
still a viscosity supersolution).

In our setting,
the weak supersolution framework is recovered thanks to the forthcoming Lemma~\ref{LEM:AP},
which can be seen as the viscosity counterpart of~\cite[Lemma~6.3]{MR3934589}, and
which in turn leads to Proposition~\ref{LANSCDPOJFIBFerngvny83iurjf}.

\begin{lemma}\label{LEM:AP}
Let~${\mathcal{R}}\subset\R^n$ and~$S:= \overline{{\mathcal{R}}}\setminus {\mathcal{R}}$. Assume that~${\mathcal{R}}$ is locally a $C^2$-smooth hypersurface and that~\eqref{o3ePP-1-pre} and~\eqref{o3ePP-1}
are satisfied.

Let~$u:{\mathcal{R}}\to[0,+\infty)$ be H\"older continuous
and such that, for every~$x\in {\mathcal{R}}$,
\begin{equation}\label{EQ-01} \int_{\mathcal{R}} \frac{u(x)-u(y)}{|x-y|^{n+s}}\,d{\mathcal{H}}^{n-1}_y\ge0\end{equation}
in the viscosity sense.

Given~$Q>0$, let~$u_Q:=\min\{u,Q\}$.

Then, for every function~$\zeta$ with support of~$\zeta\big|_{\mathcal{R}}$ bounded and contained in~${\mathcal{R}}$ and such that\begin{equation}\label{EQ-03.dwdv}
\iint_{{\mathcal{R}}\times {\mathcal{R}}} \frac{|\zeta(x)-\zeta(y)|^2}{|x-y|^{n+s}}\,d{\mathcal{H}}^{n-1}_x\,d{\mathcal{H}}^{n-1}_y<+\infty,\end{equation} we have that
\begin{equation}\label{EQ-03}
\iint_{{\mathcal{R}}\times {\mathcal{R}}} \frac{(u_Q(x)-u_Q(y))(\zeta(x)-\zeta(y))}{|x-y|^{n+s}}\,d{\mathcal{H}}^{n-1}_x\,d{\mathcal{H}}^{n-1}_y \ge0.
\end{equation}

Moreover, for every~$R_0>0$,
\begin{equation}\label{ORA:2.4}
\iint_{({\mathcal{R}}\cap B_{R_0})\times {\mathcal{R}}} \frac{|u_Q(x)-u_Q(y)|^2}{|x-y|^{n+s}}\,d{\mathcal{H}}^{n-1}_x\,d{\mathcal{H}}^{n-1}_y <+\infty
\end{equation}
and, for every function~$\phi:\R^n\to[0,+\infty)$ such that
$$\iint_{{\mathcal{R}}\times {\mathcal{R}}} \frac{|\phi(x)-\phi(y)|^2}{|x-y|^{n+s}}\,d{\mathcal{H}}^{n-1}_x\,d{\mathcal{H}}^{n-1}_y<+\infty,$$
we have that
\begin{equation}\label{ORA:2.5}
\iint_{{\mathcal{R}}\times {\mathcal{R}}} \frac{(u_Q(x)-u_Q(y))(\phi(x)-\phi(y))}{|x-y|^{n+s}}\,d{\mathcal{H}}^{n-1}_x\,d{\mathcal{H}}^{n-1}_y \ge0.
\end{equation}
\end{lemma}

\begin{proof}
The desired results will follow via a locally Lipschitz approximation.
Namely, we define~${\mathcal{R}}_k$ the set of points in~${\mathcal{R}}$ at distance at least~$\frac2k$ from~$S$ and we claim that \begin{equation}\label{ne53.23}\begin{split}&
{\mbox{there exists a sequence of locally Lipschitz functions~$u_k$ on~${\mathcal{R}}_k$}}\\
&{\mbox{that at each point of~${\mathcal{R}}_k$ has locally a smooth tangent function by above,}}\\
&{\mbox{with~$u_k$ converging to~$u_Q$
locally uniformly in~${\mathcal{R}}$,}}\\&{\mbox{and such that, for every~$x\in {\mathcal{R}}_k$,}}\\&
\int_{\mathcal{R}} \frac{u_k(x)-u_k(y)}{|x-y|^{n+s}}\,d{\mathcal{H}}^{n-1}_y\ge0.\end{split}\end{equation}
We stress that the last line of~\eqref{ne53.23} is intended in the pointwise sense
(the left-hand side being possibly equal to~$+\infty$, but not~$-\infty$, due to the smooth touching from above).

To prove~\eqref{ne53.23}, we argue as follows.
First of all, we can suppose that~$u_Q$ is not constant, otherwise we are done.
Then, given~$k\in\N$, we can replace~$u_Q$ with the function
$$\begin{dcases}
\sup_{\mathcal{R}}u_Q-\displaystyle\frac1k & {\mbox{ if }} u_Q(x)>\sup_{\mathcal{R}}u_Q-\displaystyle\frac1k,\\
u_Q(x)&{\mbox{ if }}u_Q(x)\le\sup_{\mathcal{R}}u_Q-\displaystyle\frac1k.
\end{dcases}$$
Up to this initial approximation and thanks to~\eqref{EQ-01}, we can suppose that~$u_Q$ is a strict supersolution, namely, for every~$x\in {\mathcal{R}}$,
\begin{equation}\label{EQ-01-lanu395} \int_{\mathcal{R}} \frac{u_Q(x)-u_Q(y)}{|x-y|^{n+s}}\,d{\mathcal{H}}^{n-1}_y>0\end{equation}
in the viscosity sense.

Now, given~$k$, we consider the family~${\mathcal{F}}_k$
of balls~$B$ of radius~$\frac1k$
whose closure lies at distance~$\frac1k$ from~$S$.
For every~$B\in{\mathcal{F}}_k$,
we define~$v_B$ the $s$-harmonic replacement of~$u_Q$ in~$B$,
with~$v_B=u_Q$ outside~$B$. Then, we set
\begin{equation}\label{Deefjmdcuk} u_k(x):=\inf_{B\in{\mathcal{F}}_k}v_B(x).\end{equation}
By construction, we have that~$v_B\le u_Q$, $v_B$ is locally  H\"older continuous
with H\"older constant depending only on~${\mathcal{R}}$, $n$, $s$ and~$k$,
and~$v_B$ is $s$-superharmonic in the viscosity sense, hence the same holds true for~$u_k$.
Actually, the H\"older bound on~$v_B$ guarantees that the infimum in~\eqref{Deefjmdcuk} is attained
and we pick~$B_x\in{\mathcal{F}}_k$ such that~$u_k(x)=v_{B_x}(x)$.

Note that, without loss of generality
\begin{equation}\label{sd0cjoXqowdjfcvC}
x\in B_x,
\end{equation}
otherwise
$$u_Q(x)=u_k(x)=v_{B_x}(x)=\inf_{B\in{\mathcal{F}}_k}v_B(x)\le v_{B_{\frac1k}(x)}(x)\le u_Q(x),$$
hence we could take~$B_x:=B_{\frac1k}(x)$ and obtain~\eqref{sd0cjoXqowdjfcvC}.

Thus, by~\eqref{sd0cjoXqowdjfcvC}, we pick a smooth function~$f$ touching~$v_{B_x}$ from above at~$x$ and we have that~$f\ge v_{B_x}\ge u_k$,
with equalities at~$x$, showing that,
at each point of~${\mathcal{R}}_k$, the function~$u_k$ has locally a smooth tangent function by above.

Hence, if~$f$ is a smooth function touching~$u_k$ at~$x\in {\mathcal{R}}_k$, the viscosity solution property of~$u_k$ gives that
$$0\le\int_{\mathcal{R}} \frac{f(x)-f(y)}{|x-y|^{n+s}}\,d{\mathcal{H}}^{n-1}_y\le
\int_{\mathcal{R}} \frac{u_k(x)-u_k(y)}{|x-y|^{n+s}}\,d{\mathcal{H}}^{n-1}_y,$$
as claimed in~\eqref{ne53.23}.

Furthermore, by scaling, we have that~$v_{B_x}$ is H\"older continuous for some H\"older exponent~$\alpha\in(0,1)$, with H\"older norm bounded in dependence only of~${\mathcal{R}}$, $n$ and~$s$. 
This gives that, if~$q\in\partial B_x$,
\begin{eqnarray*}&& u_Q(x)\ge u_k(x)=v_{B_x}(x)\ge v_{B_x}(q)-C|x-q|^\alpha\\&&\quad=u_{Q}(q)-C|x-q|^\alpha\ge
u_{Q}(x)-C|x-q|^\alpha\ge u_Q(x)-\frac{C}{k^\alpha},\end{eqnarray*}
up to renaming~$C$, and this entails that~$u_k$ converges to~$u_Q$
locally uniformly in~${\mathcal{R}}$.

It remains to show that, for a given~$k\in\N$, the function~$u_k$ is locally Lipschitz.
To this end, it suffices to check that for every~$x\in{\mathcal{R}}_k$
there exists~$\rho_x>0$ such that~$u_k$ is Lipschitz in~$B_{\rho_x}(x)$
(from this, the locally Lipschitz property follows by a finite covering argument with balls of radius~$\frac{\rho_x}{2}$).

For this, up to exchanging~$y$ and~$z$, it is enough to check that when~$y$, $z\in B_{\rho_x}(x)$, we have that
\begin{equation}\label{0udojf07nvb7627bjojlfFAGerwthu4}
u_k(y)-u_k(z)\le C_x|y-z|,
\end{equation}
for some~$C_x>0$.

To check this property, we employ~\eqref{EQ-01-lanu395} and
the maximum principle for nonlocal equations to see that~$u_k(x)=v_{B_x}(x)<u_Q(x)$. Hence, we deduce from the H\"older continuity
of these functions that there exists~$\rho_x>0$ such that
\begin{equation}\label{0udojf07nvb7627bjojlfFAGerwthu}
{\mbox{$u_k(p)<u_Q(p)$ for all~$p\in B_{2\rho_x}(x)$.}}\end{equation}
Then we claim that
\begin{equation}\label{0udojf07nvb7627bjojlfFAGerwthu2}
{\mbox{there exists~$\delta_x\in(0,\rho_x)$ such that, for all~$p\in B_{\rho_x}(x)$, we have that~$B_{\delta_x}(p)\subseteq B_p$.}}
\end{equation}
Indeed, if not, there would exist a sequence of points~$p_j\in B_{\rho_x}(x)$ such that~$p_j\to p_\infty$
as~$j\to+\infty$, with~$u_Q(p_\infty)=u_k(p_\infty)$.
Since~$p_\infty\in\overline{B_{\rho_x}(x)}\subseteq B_{2\rho_x}(x)$, we have reached a contradiction with~\eqref{0udojf07nvb7627bjojlfFAGerwthu} and the proof of~\eqref{0udojf07nvb7627bjojlfFAGerwthu2}
is complete.

Now, as above, given~$y$, $z\in B_{\rho_x}(x)$, we touch~$v_{B_z}$ at~$z$ with a smooth function~$f_z$ by above in~$B_z$ and, thanks to~\eqref{0udojf07nvb7627bjojlfFAGerwthu2}, we can suppose that~$f_z$ is above~$u_k$
in~$B_{2\rho_x}(x)$. We also denote by~$L_z$ the Lipschitz constant of~$f_z$ and we stress that,
in view of~\eqref{0udojf07nvb7627bjojlfFAGerwthu2} and the interior regularity for nonlocal equations,
we can suppose that~$L_z$ is bounded by a quantity~$C_x$ depending only on~${\mathcal{R}}$, $n$, $s$, and~$\delta_x$.

This yields that
$$ u_k(y)\le f_z(y)\le f_z(z)+L_z|y-z|=u_k(z)+L_z|y-z|\le u_k(z)+C_x|y-z|,$$
which establishes~\eqref{0udojf07nvb7627bjojlfFAGerwthu4}.

The proof of~\eqref{ne53.23} is thereby complete.

We now use~\eqref{ne53.23} to see that, if~$\zeta\in C^2(\R^n,[0,+\infty))$ with support of~$\zeta\big|_{\mathcal{R}}$ bounded and contained in~${\mathcal{R}}$,
\begin{equation}\label{ne53.2.x3}\begin{split}
&\iint_{{\mathcal{R}}\times {\mathcal{R}}} \frac{(u_k(x)-u_k(y))(\zeta(x)-\zeta(y))}{|x-y|^{n+s}}\,d{\mathcal{H}}^{n-1}_x\,d{\mathcal{H}}^{n-1}_y\\&\qquad=\lim_{\e\searrow0}
\iint_{{{\mathcal{R}}\times {\mathcal{R}}}\atop{\{|x-y|\ge\e\}}} \frac{(u_k(x)-u_k(y))(\zeta(x)-\zeta(y))}{|x-y|^{n+s}}\,d{\mathcal{H}}^{n-1}_x\,d{\mathcal{H}}^{n-1}_y\\&\qquad=2\lim_{\e\searrow0}
\iint_{{{\mathcal{R}}\times {\mathcal{R}}}\atop{\{|x-y|\ge\e\}}} \frac{(u_k(x)-u_k(y))\zeta(x)}{|x-y|^{n+s}}\,d{\mathcal{H}}^{n-1}_x\,d{\mathcal{H}}^{n-1}_y\\&\qquad\ge0.\end{split}
\end{equation}

Now we take~$\eta\in C^\infty_0(B_{R_0+1},[0,1])$ with~$\eta=1$ in~$B_{R_0}$ and~$|\nabla\eta|\le2$.
We test the equation in~\eqref{ne53.23} against~$(Q+1-u_k)\eta^2$, which is nonnegative in~$B_{R_0+1}$ when~$k$ is sufficiently large,
and we find that, for some~$C>0$ possibly varying from line to line,
\begin{eqnarray*}
0&\le&2\iint_{\mathcal{R}\times{\mathcal{R}}} \frac{\big(u_k(x)-u_k(y)\big) \,\big(Q+1-u_k(x)\big)\,\eta^2(x)}{|x-y|^{n+s}}\,d{\mathcal{H}}^{n-1}_y\,d{\mathcal{H}}^{n-1}_x\\&=&
\iint_{\mathcal{R}\times{\mathcal{R}}} \frac{\big(u_k(x)-u_k(y)\big) \,\big(Q+1-u_k(y)\big)\,\big(\eta(x)
+\eta(y)\big)\big(\eta(x)-\eta(y)\big)
}{|x-y|^{n+s}}\,d{\mathcal{H}}^{n-1}_y\,d{\mathcal{H}}^{n-1}_x\\&&\qquad-
\iint_{\mathcal{R}\times{\mathcal{R}}} \frac{\big|u_k(x)-u_k(y)\big|^2\,\eta^2(x)}{|x-y|^{n+s}}\,d{\mathcal{H}}^{n-1}_y\,d{\mathcal{H}}^{n-1}_x\\&\le&
C\iint_{\mathcal{R}\times{\mathcal{R}}} \frac{\big|Q+1-u_k(y)\big|^2\,\big|\eta(x)-\eta(y)\big|^2
}{|x-y|^{n+s}}\,d{\mathcal{H}}^{n-1}_y\,d{\mathcal{H}}^{n-1}_x\\&&\qquad-\frac12
\iint_{\mathcal{R}\times{\mathcal{R}}} \frac{\big|u_k(x)-u_k(y)\big|^2\,\eta^2(x)}{|x-y|^{n+s}}\,d{\mathcal{H}}^{n-1}_y\,d{\mathcal{H}}^{n-1}_x\\&\le&
C\,(Q+1)^2-\frac12
\iint_{(\mathcal{R}\cap B_{R_0})\times{\mathcal{R}}} \frac{\big|u_k(x)-u_k(y)\big|^2}{|x-y|^{n+s}}\,d{\mathcal{H}}^{n-1}_y\,d{\mathcal{H}}^{n-1}_x.
\end{eqnarray*}
The claim in~\eqref{ORA:2.4} now follows by sending~$k\to+\infty$.

As a byproduct of~\eqref{ORA:2.4} we also have that, for all $C^2$ and compactly supported functions~$\psi$,
\begin{equation*}
\begin{split}&\lim_{k\to+\infty}
\iint_{{\mathcal{R}}\times {\mathcal{R}}} \frac{(u_k(x)-u_k(y))(\psi(x)-\psi(y))}{|x-y|^{n+s}}\,d{\mathcal{H}}^{n-1}_x\,d{\mathcal{H}}^{n-1}_y
\\&=2\lim_{k\to+\infty}
\iint_{{\mathcal{R}}\times {\mathcal{R}}} \frac{u_k(x)\,(\psi(x)-\psi(y))}{|x-y|^{n+s}}\,d{\mathcal{H}}^{n-1}_x\,d{\mathcal{H}}^{n-1}_y\\&=2
\iint_{{\mathcal{R}}\times {\mathcal{R}}} \frac{u_Q(x)\,(\psi(x)-\psi(y))}{|x-y|^{n+s}}\,d{\mathcal{H}}^{n-1}_x\,d{\mathcal{H}}^{n-1}_y\\&=\iint_{{\mathcal{R}}\times {\mathcal{R}}} \frac{(u_Q(x)-u_Q(y))(\psi(x)-\psi(y))}{|x-y|^{n+s}}\,d{\mathcal{H}}^{n-1}_x\,d{\mathcal{H}}^{n-1}_y.
\end{split}
\end{equation*}
Thanks to this observation, we can now pass to the limit as~$k\to+\infty$ in~\eqref{ne53.2.x3}
and obtain~\eqref{EQ-03} for~$\zeta\in C^2(\R^n,[0,+\infty))$, 
and thus, by density (see e.g.~\cite[Theorem~2.7]{MR1481970}), for~$\zeta$ as in~\eqref{EQ-03.dwdv},
as desired.

It remains to prove~\eqref{ORA:2.5}. To this end, we use~\eqref{o3ePP-1-pre} and~\eqref{o3ePP-1}
to follow a capacity argument in~\cite[Proposition~3.7 and equation~(6.11)]{MR3934589}. 
Namely, we find a sequence of
functions~$w_{j}$ with values in~$[0,1]$, with~$w_j=1$ in an open neighborhood of~$S$,
converging to zero in~${\mathcal{R}}$, up to a set of null~$(n-1)$-dimensional
Hausdorff measure, and such that
\begin{equation}\label{ORA:2.4o12ehdfnv}
\lim_{j\to+\infty}\iint_{{\mathcal{R}}\times{\mathcal{R}}} \frac{|w_j(x)-w_j(y)|^2}{|x-y|^{n+s}}\,d{\mathcal{H}}^{n-1}_x\,d{\mathcal{H}}^{n-1}_y=0.\end{equation}
Thus, given~$M>0$, we use~\eqref{EQ-03} with~$\zeta:=W_j\phi_M$, with~$\phi_M:=\min\{\phi,M\}$ and~$
W_j:=1-w_j$, and we see that
\begin{equation}\label{ORA:2.4o12ehdfnv.3}\begin{split}
0&\le
2\iint_{{\mathcal{R}}\times {\mathcal{R}}} \frac{\big(u_Q(x)-u_Q(y)\big)\big(W_j(x)\phi_M(x)-W_j(y)\phi_M(y)\big)}{|x-y|^{n+s}}\,d{\mathcal{H}}^{n-1}_x\,d{\mathcal{H}}^{n-1}_y \\&=
\iint_{{\mathcal{R}}\times {\mathcal{R}}} \frac{\big(u_Q(x)-u_Q(y)\big)\big(\phi_M(x)-\phi_M(y)\big)\,W_j(x)}{|x-y|^{n+s}}\,d{\mathcal{H}}^{n-1}_x\,d{\mathcal{H}}^{n-1}_y \\&\quad
-\iint_{{\mathcal{R}}\times {\mathcal{R}}} \frac{\big(u_Q(x)-u_Q(y)\big)\big(w_j(x)-w_j(y)\big)\,\phi_M(y)}{|x-y|^{n+s}}\,d{\mathcal{H}}^{n-1}_x\,d{\mathcal{H}}^{n-1}_y .
\end{split}\end{equation}
Besides, by~\eqref{ORA:2.4} and the Dominated Convergence Theorem,
\begin{equation}\label{ORA:2.4o12ehdfnv.2}\begin{split}&
\lim_{j\to+\infty}\iint_{{\mathcal{R}}\times {\mathcal{R}}} \frac{\big(u_Q(x)-u_Q(y)\big)\big(\phi_M(x)-\phi_M(y)\big)\,W_j(x)}{|x-y|^{n+s}}\,d{\mathcal{H}}^{n-1}_x\,d{\mathcal{H}}^{n-1}_y\\&\qquad=\iint_{{\mathcal{R}}\times {\mathcal{R}}} \frac{\big(u_Q(x)-u_Q(y)\big)\big(\phi_M(x)-\phi_M(y)\big)}{|x-y|^{n+s}}\,d{\mathcal{H}}^{n-1}_x\,d{\mathcal{H}}^{n-1}_y.
\end{split}\end{equation}

On a similar note, recalling~\eqref{ORA:2.4} and~\eqref{ORA:2.4o12ehdfnv},
$$\lim_{j\to+\infty}\iint_{{\mathcal{R}}\times {\mathcal{R}}} \frac{\big(u_Q(x)-u_Q(y)\big)\big(w_j(x)-w_j(y)\big)\,\phi_M(y)}{|x-y|^{n+s}}\,d{\mathcal{H}}^{n-1}_x\,d{\mathcal{H}}^{n-1}_y =0.$$

Thanks to this and~\eqref{ORA:2.4o12ehdfnv.2}, we can pass to the limit as~$j\to+\infty$ in~\eqref{ORA:2.4o12ehdfnv.3} and conclude that $$\iint_{{\mathcal{R}}\times {\mathcal{R}}} \frac{\big(u_Q(x)-u_Q(y)\big)\big(\phi_M(x)-\phi_M(y)\big)}{|x-y|^{n+s}}\,d{\mathcal{H}}^{n-1}_x\,d{\mathcal{H}}^{n-1}_y\ge0.$$
This gives~\eqref{ORA:2.5}, after sending~$M\to+\infty$ and recalling again~\eqref{ORA:2.4},
as desired.\end{proof}

\section{H\"older estimates for fractional operators in a geometric setting}\label{LKMHOLE}

Here we present a H\"older regularity result in a geometric setting.
Namely, differently from the cases already treated in the literature, the equation considered here is defined by an integral on a portion of a hypersurface and the kernel is not necessarily symmetric,  but only symmetric up to a suitable remainder. 

\begin{lemma}\label{DERGED-Ge}
Let~${\mathcal{R}}$ be a portion of a $C^2$ hypersurface in~$B_2$ 
that is $C^2$-diffeomorphic to~$B_2\cap\{x_n=0\}$.
Let~$\nu$ be the unit normal vector field of~${\mathcal{R}}$ and~$f\in L^\infty(B_2)$.

Let~$K:{\mathcal{R}}\times{\mathcal{R}}\to[0,+\infty]$ be such that
\begin{equation}\label{jqowdfendP2}\begin{split}&
{\mbox{for all~$x\ne x_0$, we have that }}\;
\frac{1}{C|x-x_0|^{n+s}} \leq K(x,x_0)\le \frac{C}{|x-x_0|^{n+s}},\\ &{\mbox{and, for all~$y\in B_\delta\setminus\{0\}$, }}
|K(y_{x_0}^+,x_0)-K(y_{x_0}^-,x_0)|\le C|y|^{1-n-s},
\end{split}\end{equation}
for some (small) $\delta\in(0,1)$ and some~$C\ge1$.

Assume that, for all~$x_0\in{{\mathcal{R}}}\cap B_1$, the function~$v\in L^\infty({\mathcal{R}})$
is a solution of
$$\int_{{\mathcal{R}}}
\big( v(x_0)-v(x)\big)\,K(x,x_0)
\,d{\mathcal{H}}^{n-1}_x=f(x_0).
$$

Then,  $v$ is H\"older continuous in~${{\mathcal{R}}}\cap B_{1/2}$. More precisely,
\begin{equation}\label{09ugh0oi90pot00-ikjmrflxP} \|v\|_{C^\alpha({\mathcal{R}}\cap B_{1/2})}\le C_0\Big(
\|f\|_{L^\infty(B_2)}+ \|v\|_{L^\infty({\mathcal{R}})}\Big),\end{equation}
where $\alpha\in(0,1)$ and~$C_0>0$ depend only on~$n$, $s$, the regularity parameters of~${\mathcal{R}}$ and the structural constant~$C$ of the kernel~$K$ in~\eqref{jqowdfendP2}.
\end{lemma}

\begin{proof} We let~$u:=v\chi_{B_{3/4}}$ and we see that, for all~$x_0\in{{\mathcal{R}}}\cap B_{2/3}$,
\begin{eqnarray*}&&
\int_{{\mathcal{R}}}
\big( u(x_0)-u(x)\big)\,K(x,x_0)
\,d{\mathcal{H}}^{n-1}_x=\int_{{\mathcal{R}}}
\big( v(x_0)-v(x)\chi_{B_{3/4}}(x)\big)\,K(x,x_0)
\,d{\mathcal{H}}^{n-1}_x\\&&\qquad=f(x_0)+\int_{{\mathcal{R}}\setminus B_{3/4}}
v(x)\,K(x,x_0)
\,d{\mathcal{H}}^{n-1}_x=:g(x_0)
\end{eqnarray*}
and
\begin{eqnarray*}
&& |g(x_0)|\le \|f\|_{L^\infty(B_2)}+C\|v\|_{L^\infty({\mathcal{R}})}
\int_{{\mathcal{R}}\setminus B_{3/4}}\frac{d{\mathcal{H}}^{n-1}_x}{|x-x_0|^{n+s}}\\&&\qquad\qquad\qquad
\le C\Big(  \|f\|_{L^\infty(B_2)}+ {\mathcal{H}}^{n-1}({\mathcal{R}})\|v\|_{L^\infty({\mathcal{R}})}\Big)=:C_\star,
\end{eqnarray*}
up to renaming~$C$ from line to line.

We now take~$B_2':=B_2\cap\{x_n=0\}$ and a $C^2$-diffeomorphism~$\phi:{\mathcal{R}}\to B'_2$. We can also adapt the diffeomorphism so that~$\phi$ coincides with the projection in a small neighborhood of~$x_0$
and~$\phi({\mathcal{R}}\cap B_{2/3})$ contains~$B'_{3/4}$.
We define~$U(y):=u(\phi^{-1}(y))$ and we observe that, for all~$y_0=\phi(x_0)\in B'_{3/4}$,
\begin{equation}\label{02ojfe22}\begin{split}& G(y_0):=g(\phi^{-1}(y_0))=
\int_{{\mathcal{R}}}
\big( u(\phi^{-1}(y_0))-u(x)\big)\,K(x,\phi^{-1}(y_0))
\,d{\mathcal{H}}^{n-1}_x\\&\qquad=\int_{B'_2}
\big( U(y_0)-U(y)\big)\,K_*(y,y_0)
\,d{\mathcal{H}}^{n-1}_y,\end{split}
\end{equation}
where
\begin{equation}\label{02ojfe22-2}
K_*(y,y_0):=K(\phi^{-1}(y),\phi^{-1}(y_0))\,J(y),\end{equation} for a suitable Jacobian function~$J$, which is bounded and bounded away from zero.

Thus, by~\eqref{jqowdfendP2}, we have that~$K_*(y,y_0)$ is bounded from above and below, up to constants, by~$|\phi^{-1}(y)-\phi^{-1}(y_0)|^{-(n+s)}$,
which in turn is comparable to~${|y-y_0|^{-(n+s)}}$, since~$\phi$ is a diffeomorphism.

Additionally, if
$$ L(y_0,z):=K_*(y_0+z,y_0)-K_*(y_0-z,y_0),$$
we have that
\begin{equation}\label{SkPALmzyHS20}\begin{split}
|L(y_0,z) |&\le
|K(\phi^{-1}(y_0+z),\phi^{-1}(y_0))-K(\phi^{-1}(y_0-z),\phi^{-1}(y_0))|\,J(y_0+z)
\\&\qquad\qquad\qquad
+|K(\phi^{-1}(y_0-z),\phi^{-1}(y_0))|\,|J(y_0+z)-J(y_0-z)|\\
&\le C |K(\phi^{-1}(y_0+z),\phi^{-1}(y_0))-K(\phi^{-1}(y_0-z),\phi^{-1}(y_0))|
\\&\qquad\qquad\qquad+C|z|^{-n-s}\,|J(y_0+z)-J(y_0-z)|\\&
\le C|z|^{1-n-s},\end{split}\end{equation}
up to renaming~$C$ line after line.

Actually, up to renaming~$G$ by an additional bounded function, if we extend~$U$ by zero outside~$B'_2$ we can rewrite~\eqref{02ojfe22} in the form
\begin{equation}\label{0iqwjdfle-3tyy} \int_{\R^{n-1}}
\big( U(y_0)-U(y)\big)\,K_*(y,y_0)\,d{\mathcal{H}}^{n-1}_y=G(y_0),\end{equation}
for all~$y_0\in B_{3/4}'$.

We can thereby apply the regularity theory for integro-differential equations
(see e.g.~\cite[Theorem 5.4 and Remark 4.4]{MR2244602}, or~\cite[Theorem~3.1]{MR2995098},
or~\cite[Theorem~1.1(ii)]{KW})
and obtain the desired H\"older estimate for~$U$,
and therefore for~$u$, and then for~$v$.
\end{proof}

\section{An upper bound in a geometric setting}\label{LKMHOLE-1}

In the proof of our main result, as it will be presented in Section~\ref{HAR:NA},
a somewhat delicate issue comes from the possibility of extracting a convergent sequence
from the renormalized distance between minimal sheets. Roughly speaking, the plan would
be to remove an arbitrarily small neighborhood of the singular set and focus on an arbitrarily large ball,
obtain estimates in this domain which are uniform with respect to the sequences under consideration
(with constants possibly depending on the neighborhood of the singular set and
on the radius of the large ball),
pass the sequence to the limit, and only at the end of the argument
shrink the neighborhood of the singular set and invade the space by larger and larger balls.

In this setting, Lemma~\ref{DERGED-Ge} is instrumental to provide the necessary compactness.
However, to use this result, one needs a bound in~$L^\infty$, as dictated by the right-hand side of~\eqref{09ugh0oi90pot00-ikjmrflxP}.

Such a bound does not come completely for free, not even in the situation in which
the distance between minimal sheets is normalized to be~$1$ at some point,
due to the possible divergence of this function at the singular set and the correspondingly large
values that this function could, in principle, attain in the domain under consideration. To get around such a difficulty,
we present here an $L^\infty$ bound in terms of an integral estimate (as made precise by~\eqref{XM-skdadM-BIS} here below).
This integral control will be then checked in our specific case via a sliding method (as made precise by the forthcoming Lemma~\ref{ZX-ipfe-jJSL}).

This strategy will thus allow us to pass to the limit locally away from the singular set: summarizing,
one needs to just establish an integral bound (that will come from Lemma~\ref{ZX-ipfe-jJSL}), to deduce from it a bound in~$L^\infty$ (coming from the next result),
and thus obtain a bound in~$C^\alpha$ (coming from Lemma~\ref{DERGED-Ge}), which in turns provides the desired compactness property (by the Arzel\`a-Ascoli Theorem).

\begin{lemma}\label{VOL}
Let~${\mathcal{R}}$ be a portion of a $C^2$ hypersurface in~$B_2$ 
that is $C^2$-diffeomorphic to~$B_2\cap\{x_n=0\}$.
Let~$\nu$ be the unit normal vector field of~${\mathcal{R}}$.

Let~$K:{\mathcal{R}}\times{\mathcal{R}}\to[0,+\infty]$ be such that
\begin{equation}\label{IPSJLkear34FPA}\begin{split}&
{\mbox{for all~$x\ne x_0$, we have that }}\;
\frac{1}{C|x-x_0|^{n+s}} \leq K(x,x_0)\le \frac{C}{|x-x_0|^{n+s}},\\ &{\mbox{and, for all~$y\in B_\delta\setminus\{0\}$, }}
|K(y_{x_0}^+,x_0)-K(y_{x_0}^-,x_0)|\le C|y|^{1-n-s},
\end{split}\end{equation}
for some (small) $\delta\in(0,1)$ and some~$C\ge1$.

Assume that, for all~$x_0\in{{\mathcal{R}}}\cap B_1$, the function~$v$ satisfies
\begin{equation}\label{XM-skdadM}\begin{split}&\int_{{{\mathcal{R}}}}
\big( v(x_0)-v(x)\big)\,K(x,x_0)
\,d{\mathcal{H}}^{n-1}_x-a(x_0)v(x_0)\le M,\end{split}\end{equation}
for some~$M\ge0$ and~$a\in L^\infty(\R^n)$,
and that
\begin{equation}\label{XM-skdadM-BIS} \int_{{\mathcal{R}}}v^+(x)\,d{\mathcal{H}}^{n-1}_x\le M',\end{equation}
for some~$M'\ge0$.

Then, $v$ is bounded from above in~${\mathcal{R}}\cap B_{1/2}$, with
$$ \sup_{{\mathcal{R}}\cap B_{1/2}}v\le C_0
\left(M+M'\right)
,$$
with~$C_0>0$ depending only on~$n$, $s$, $\|a\|_{L^\infty(\R^n)}$,
the regularity parameters of~${\mathcal{R}}$ and the structural constant~$C$ of the kernel~$K$ in~\eqref{IPSJLkear34FPA}.
\end{lemma}

\begin{proof} The argument presented here extends to the geometric framework and for more general kernels a method of proof utilized in~\cite[Lemma~5.2]{MR3661864}. Specifically, 
recalling the notation~$B'_r:=B_r\cap\{x_n=0\}$,
as in~\eqref{02ojfe22-2} we consider a $C^2$-diffeomorphism~$\phi:{\mathcal{R}}\to B'_2$,
define~$V(y):=v(\phi^{-1}(y))$ and look at the corresponding integral equation driven by the kernel~$K_*$

We look at the function~$B'_1\ni y\mapsto\Psi(y):=\Theta\,(1-|y|^2)^{-n}$, with~$\Theta>0$ suitably large, and we slide the graph of~$\Psi$ by above in~$B_{1}'$ till we touch the graph of~$V$. That is, we find~$t\in\R$ such that~$V\le \Psi+t$ in~$B_1'$
with equality holding at some~$q\in B_1'$. 
We can assume that~$t\ge0$ (otherwise, $v\le\Psi+t\le\Psi$ and this would give the desired bound).

Let also~$d:=1-|q|$.
Let~$r\in\left(0,\frac{d}{2}\right)$ to be taken suitably small.
We consider the tangent plane of the barrier~$\Psi+t$ at~$q$, namely the linear function
$$ \ell(y):=\nabla\Psi(q)\cdot(y-q).
$$
It would be desirable to freely subtract tangent planes in the operators, but this is not possible in our framework, due to the lack of symmetry of the kernel, therefore we need to take care of an additional remainder. Namely, 
\begin{eqnarray*}&&\left|
2\int_{B'_r(q)}\ell(y)\,K_*(y,q)\,d{\mathcal{H}}^{n-1}_y\right|
=\left|2\nabla\Psi(q)\cdot\int_{B_r'(q)}(y-q)\,K_*(y,q)\,d{\mathcal{H}}^{n-1}_y\right|\\&&\qquad=
\left|\nabla\Psi(q)\cdot\int_{B_r'}z\big(K_*(q+z,q)-K_*(q-z,q)\big)
\,d{\mathcal{H}}^{n-1}_z\right|
.
\end{eqnarray*}

Now, as observed in~\eqref{SkPALmzyHS20}, it follows from~\eqref{IPSJLkear34FPA} that
\begin{eqnarray*}
|K_*(q+z,q)-K_*(q-z,q)|\le C|z|^{1-n-s}
\end{eqnarray*}
As a result,
\begin{equation*}\begin{split}&\left|
2\int_{B'_r(q)}\ell(y)\,K_*(y,q)\,d{\mathcal{H}}^{n-1}_y\right|\le
|\nabla\Psi(q)| \,\int_{B_r'}|z|\,\big|K_*(q+z,q)-K_*(q-z,q)\big|
\,d{\mathcal{H}}^{n-1}_z
\\&\qquad\qquad\le
C|\nabla\Psi(q)|\,\int_{B_r'}|z|^{2-n-s}\,d{\mathcal{H}}^{n-1}_z
\le
C|\nabla\Psi(q)|\,r^{1-s}.
\end{split}\end{equation*}

Consequently, if~$q:=\phi(p)$ and~${\mathcal{R}}_r(p):=\phi^{-1}(B_r'(q))$,
\begin{equation}\label{0qwdj3toftwta}\begin{split}&
\int_{{\mathcal{R}}_r(p)}
\big( v(p)-v(x)\big)\,K(x,p)
\,d{\mathcal{H}}^{n-1}_x\\&\qquad=
\int_{B_r'(q)}
\big( V(q)-V(y)+\ell(y)\big)\,K_*(y,q)
\,d{\mathcal{H}}^{n-1}_y-
\int_{B'_r(q)}
\ell(y)\,K_*(y,q)
\,d{\mathcal{H}}^{n-1}_y\\&\qquad\ge
-\int_{B'_r(q)}
\big( V(y)-V(q)-\ell(y)\big)\,K_*(y,q)
\,d{\mathcal{H}}^{n-1}_x-C|\nabla\Psi(q)|\,r^{1-s}.
\end{split}
\end{equation}
Now we use the notation of positive and negative parts of a function, namely~$g=g_+-g_-$, where~$g_+:=\max\{g,0\}$ and~$g_-:=\max\{-g,0\}$, to see that, if~$y\in B'_r(q)$,
$$ |y|\le |q|+|y-q|\le 1-d+r\le 1-\frac{d}{2}$$
and, as a consequence,
\begin{eqnarray*}
\big( V(y)-V(q)-\ell(y)\big)_+ &\le&\max\{\Psi(y)-\Psi(q)-\ell(y),0 \}
\\&\le&\frac{C\Theta|y-q|^2}{d^{n+2}}.
\end{eqnarray*}
For this reason,
\begin{equation}\label{hidwaiuvfcvvA}\begin{split}&
\int_{B'_r(q)}\big( V(y)-V(q)-\ell(y)\big)_+\,K_*(y,q)
\,d{\mathcal{H}}^{n-1}_y\le\frac{ C\Theta}{d^{n+2}}\int_{B'_r(q)}
|y-q|^2\,K_*(y,q)
\,d{\mathcal{H}}^{n-1}_y\\&\qquad\qquad\qquad
\le\frac{ C\Theta}{d^{n+2}}\int_0^r \frac{d\rho}{\rho^s}=\frac{ C\Theta r^{1-s}}{d^{n+2}}.\end{split}
\end{equation}

Moreover,
\begin{eqnarray*}
\big( V(y)-V(q)-\ell(y)\big)_-&\ge& V(q)-V(y)+\ell(y)\\&\ge& \Psi(q)+t-V^+(y)+\ell(y)
\end{eqnarray*}
and therefore
\begin{equation*}\begin{split}&
\int_{B'_r(q)}
\big( V(y)-V(q)-\ell(y)\big)_-\,K_*(y,q)
\,d{\mathcal{H}}^{n-1}_y\\&\qquad\ge\frac1C\int_{B_r'(q)}
\big( V(y)-V(q)-\ell(y)\big)_-\,\frac{d{\mathcal{H}}^{n-1}_y}{|y-q|^{n+s}}\\&\qquad\ge
\frac1{Cr^{n+s}}\int_{B_r'(q)}
\big( V(y)-V(q)-\ell(y)\big)_-\,d{\mathcal{H}}^{n-1}_y
\\&\qquad\ge\frac1{Cr^{n+s}}
\int_{B_r'(q)}
\big(\Psi(q)+t-V^+(y)+\ell(y)\big)\,d{\mathcal{H}}^{n-1}_y\\&\qquad\ge\frac1{Cr^{n+s}}\left(
\big(\Psi(q)+t\big)r^{n-1}
-C\int_{B_r'(q)}V^+(y)\,d{\mathcal{H}}^{n-1}_y\right)-\frac{C|\nabla\Psi(q)|}{r^{s}}
\\&\qquad\ge\frac{\Psi(q)+t}{Cr^{1+s}}
-\frac{C}{r^{n+s}}\int_{{\mathcal{R}}_r(p)}v^+(x)\,d{\mathcal{H}}^{n-1}_x-\frac{C|\nabla\Psi(q)|}{r^{s}}
.\end{split}
\end{equation*}

Note that
$$\Psi(q)=\Theta\,(1-|q|^2)^{-n}=\Theta\,(1+|q|)^{-n}(1-|q|)^{-n}\ge 2^{-n}\Theta d^{-n}$$
and
$$ |\nabla\Psi(q)|=\frac{2n\Theta|q|}{(1-|q|^2)^{n+1}}\le \frac{2n\Theta}{d^{n+1}},$$
giving that
\begin{equation*}\begin{split}&
\int_{B_r'(q)}
\big( V(y)-V(q)-\ell(y)\big)_-\,K_*(y,q)
\,d{\mathcal{H}}^{n-1}_y\\&\qquad\ge\frac{\Theta}{Cd^n r^{1+s}}+\frac{t}{Cr^{1+s}}
-\frac{C}{r^{n+s}}\int_{{\mathcal{R}}_r(p)}v^+(x)\,d{\mathcal{H}}^{n-1}_x-\frac{C\Theta}{d^{n+1}r^{s}}
.\end{split}\end{equation*}
Notice that
\begin{eqnarray*}
\frac{\Theta}{Cd^n r^{1+s}}-\frac{C\Theta}{d^{n+1}r^{s}}
=\frac{\Theta}{Cd^{n}r^{s}}\left(\frac1{r}-\frac{C^2}{d}\right)>
\frac{\Theta}{Cd^{n}r^{1+s}}
\end{eqnarray*}
up to renaming~$C$,
if~$r$ is smaller than a small constant times~$d$,
and therefore
\begin{equation*}\begin{split}&
\int_{B_r'(q)}
\big( V(y)-V(q)-\ell(y)\big)_-\,K_*(y,q)
\,d{\mathcal{H}}^{n-1}_y\\&\qquad\ge\frac{\Theta}{Cd^n r^{1+s}}+\frac{t}{Cr^{1+s}}
-\frac{C}{r^{n+s}}\int_{{\mathcal{R}}_r(p)}v^+(x)\,d{\mathcal{H}}^{n-1}_x
.\end{split}\end{equation*}
Combining this estimate and~\eqref{hidwaiuvfcvvA}, we infer that
\begin{eqnarray*}&&
\int_{B_r'(q)}\big( V(y)-V(q)-\ell(y)\big)\,K_*(y,q)
\,d{\mathcal{H}}^{n-1}_y\\&&\qquad\le -
\frac{\Theta}{Cd^n r^{1+s}}-\frac{t}{Cr^{1+s}}
+\frac{C}{r^{n+s}}\int_{{\mathcal{R}}_r(p)}v^+(x)\,d{\mathcal{H}}^{n-1}_x
+\frac{C\Theta r^{1-s}}{d^{n+2}}
.\end{eqnarray*}
{F}rom this and~\eqref{0qwdj3toftwta} we conclude that
\begin{equation}\label{XM-skdadM2}\begin{split}&
\int_{{\mathcal{R}}_r(p)}
\big( v(p)-v(x)\big)\,K(x,p)
\,d{\mathcal{H}}^{n-1}_x\\&\qquad\ge
\frac{\Theta}{Cd^n r^{1+s}}
+\frac{t}{Cr^{1+s}}-
\frac{C}{r^{n+s}}\int_{{\mathcal{R}}_r(p)}v^+(x)\,d{\mathcal{H}}^{n-1}_x
-\frac{C\Theta r^{1-s}}{d^{n+2}},
\end{split}
\end{equation}
where, as customary, we have renamed constants line after line.

Also,
$$ v (p)=V(q)= \Psi(q)+t\ge\Psi(q)\ge0$$
and therefore
\begin{eqnarray*}&&
\int_{{\mathcal{R}}\setminus{\mathcal{R}}_r(p)}
\big( v(p)-v(x)\big)\,K(x,p)\,d{\mathcal{H}}^{n-1}_x\ge-C
\int_{{\mathcal{R}}\setminus{\mathcal{R}}_r(p)} v^+(x)\,\frac{d{\mathcal{H}}^{n-1}_x}{|x-p|^{n+s}}\\&&\qquad
\ge-\frac{C}{r^{n+s}}
\int_{{\mathcal{R}}\setminus{\mathcal{R}}_r(p)} v^+(x)\,d{\mathcal{H}}^{n-1}_x\ge-\frac{C}{r^{n+s}}
\int_{{\mathcal{R}}} v^+(x)\,d{\mathcal{H}}^{n-1}_x.
\end{eqnarray*}

This, \eqref{XM-skdadM} and~\eqref{XM-skdadM2} yield that, if $\Theta$ is large as specified above,
\begin{eqnarray*}M&\ge&
\int_{{{\mathcal{R}}}}
\big( v(p)-v(x)\big)\,K(x,p)
\,d{\mathcal{H}}^{n-1}_x-a(p)v(p)\\&\ge&\frac{\Theta}{Cd^n r^{1+s}}+\frac{t}{Cr^{1+s}}-
\frac{C}{r^{n+s}}\int_{{\mathcal{R}}}v^+(x)\,d{\mathcal{H}}^{n-1}_x-\frac{C\Theta r^{1-s}}{d^{n+2}}-a(p)(\Psi(q)+t)\\&\ge&\frac{\Theta}{Cd^n r^{1+s}}+\frac{t}{Cr^{1+s}}
-
\frac{C}{r^{n+s}}\int_{{\mathcal{R}}}v^+(x)\,d{\mathcal{H}}^{n-1}_x-\frac{C\Theta r^{1-s}}{d^{n+2}}-a(p)\Psi(q)\\&\ge&
\frac{\Theta}{Cd^n r^{1+s}}+\frac{t}{Cr^{1+s}}
-\frac{C}{r^{n+s}}\int_{{\mathcal{R}}}v^+(x)\,d{\mathcal{H}}^{n-1}_x-\frac{C\Theta r^{1-s}}{d^{n+2}}-\frac{C\Theta}{d^n}.
\end{eqnarray*}
Notice that for~$r$ small enough, the latter term can be reabsorbed.
In particular, taking~$r=\alpha d$, for a small~$\alpha\in(0,1)$, we find that
\begin{eqnarray*}M&\ge&
\frac{\Theta}{Cd^n r^{1+s}}+\frac{t}{Cr^{1+s}}
-\frac{C}{r^{n+s}}\int_{{\mathcal{R}}}v^+(x)\,d{\mathcal{H}}^{n-1}_x-\frac{C\Theta r^{1-s}}{d^{n+2}}\\
&=&\frac{\Theta}{C \alpha^{1+s} d^{n+1+s}}+\frac{t}{C\alpha^{1+s} d^{1+s}}
-\frac{C}{\alpha^{n+s} d^{n+s}}\int_{{\mathcal{R}}}v^+(x)\,d{\mathcal{H}}^{n-1}_x-\frac{C\Theta \alpha^{1-s}}{d^{n+1+s}}
\\&\ge& \frac{t}{C\alpha^{1+s} d^{1+s}}
+\frac{\Theta}{\alpha^{1+s} d^{n+1+s}}\left[\frac1C
-\frac{Cd}{\alpha^{n-1} \Theta}\int_{{\mathcal{R}}}v^+(x)\,d{\mathcal{H}}^{n-1}_x-C\alpha^{2}\right].
\end{eqnarray*}
The smallness of~$\alpha$ having played its role, we omit it from the notation from now on, absorbing it into the constants. In particular, if
$$ \Theta:= C \int_{{\mathcal{R}}}v^+(x)\,d{\mathcal{H}}^{n-1}_x,$$
with~$C$ large enough, we have that
\begin{eqnarray*}M\ge
\frac{t}{C d^{1+s}}
+\frac{\Theta}{C d^{n+1+s}}\ge \frac{t}{C d^{1+s}}.\end{eqnarray*}
All in all, we have found that, if~$\xi\in{\mathcal{R}}\cap B_{1/2}$,
\begin{eqnarray*}&&
v(\xi)=V(\phi(\xi))\le \Psi(\phi(\xi))+t\le  
C\big(M+\Theta\big)\\&&\qquad \le
C\left(M+\int_{{\mathcal{R}}}v^+(x)\,d{\mathcal{H}}^{n-1}_x\right),
\end{eqnarray*}
up to renaming~$C$.
\end{proof}

\section{An integral bound in a geometric setting}\label{INTEBO}

In retrospect, the H\"older regularity theory established in Lemma~\ref{DERGED-Ge}
relied on a specific hypothesis, namely an~$L^\infty$ bound, which in turn was reduced to an integral bound by means of Lemma~\ref{VOL}.

However, in our setting, integral bounds do not come completely for free. As already mentioned,
the complication arises from the fact that we need to apply this regularity theory to a normalized parameterization between minimal sheets: thus, in view of the possible divergence of normal parameterizations at singular points, $L^\infty$ bounds, and even $L^1$ bounds, may be nontrivial. The next result fills this gap and ensures an~$L^1$ bound, which, in our application, will lead to an~$L^\infty$ bound
(via Lemma~\ref{VOL}) and thus to a H\"older estimate (owing to Lemma~\ref{DERGED-Ge}), from which one will deduce suitable compactness properties of the normalized distance of minimal sheets.

\begin{lemma}\label{ZX-ipfe-jJSL}
Let~${\mathcal{R}}$ be a portion of a $C^2$ hypersurface
in~$B_2$ that is $C^2$-diffeomorphic to~$B_2\cap\{x_n=0\}$.
Let~$\nu$ be the unit normal vector field of~${\mathcal{R}}$.

Let~$K:{\mathcal{R}}\times{\mathcal{R}}\to[0,+\infty]$ be such that
\begin{equation}\label{ojfe-32659867}\begin{split}&
{\mbox{for all~$x\ne x_0$, we have that }}\;
\frac{1}{C|x-x_0|^{n+s}} \leq K(x,x_0)\le \frac{C}{|x-x_0|^{n+s}},\\ &{\mbox{and, for all~$y\in B_\delta\setminus\{0\}$, }}
|K(y_{x_0}^+,x_0)-K(y_{x_0}^-,x_0)|\le C|y|^{1-n-s},
\end{split}\end{equation}
for some (small) $\delta\in(0,1)$ and some~$C\ge1$.

Assume that, for all~$x_0\in{{\mathcal{R}}}\cap B_1$, the function~$v$ satisfies
\begin{equation}\label{Nkpa-ipjfw40jjo9-RFmra}\begin{split}&
\int_{\mathcal{R}}
\big( v(x_0)-v(x)\big)\,K(x,x_0)
\,d{\mathcal{H}}^{n-1}_x-a(x_0)v(x_0)\ge -M_0,\end{split}\end{equation}
for some~$M_0>0$ and~$a:\R^n\to\R$, with~$a$ bounded from below.

Assume also that~$v$ is nonnegative and that~$v(x_\star)=1$, for some~$x_\star\in{{\mathcal{R}}}\cap B_1$.

Then, $$
\int_{{\mathcal{R}}}
v(x)\,d{\mathcal{H}}^{n-1}_x\le C_0 ,$$
with~$C_0>0$ depending only on~$n$, $s$, $M_0$, $\|a_-\|_{L^\infty(\R^n)}$,
the regularity parameters of~${\mathcal{R}}$ and the structural constant~$C$ of the kernel~$K$
in~\eqref{ojfe-32659867}.
\end{lemma}

\begin{proof} We straighten~${\mathcal{R}}$ by a diffeomorphism~$\phi:{\mathcal{R}}\to B'_2:=B_2\cap\{x_n=0\}$
and set~$V(y):=v(\phi^{-1}(y))$.
We let~$y_\star:=\phi(x_\star)$ and
slide the parabola~$-M|y-y_\star|^2$
from below till we touch the graph of~$V$ at some point~$y_\sharp$. 
Notice that
$$ M|y_\sharp-y_\star|^2\le V(y_\sharp)+M|y_\sharp-y_\star|^2\le V(y_\star)=v(x_\star)=1.$$
Thus, the parameter~$M>0$ is fixed (once and for all) sufficiently large such that the touching point~$y_\sharp$
lies in~$B_{3/2}'$.

We remark that
$$ P(y):=V(y_\sharp)+M|y_\sharp-y_\star|^2-M|y-y_\star|^2\le V(y)$$
and in particular~$P(y_\sharp)=V(y_\sharp)\ge0$.

Hence, if~$\widetilde V:=V-P$, we have that~$\widetilde V(y_\sharp)=0$ and thus, recalling the notation in~\eqref{02ojfe22-2}
and setting~$x_\sharp:=\phi^{-1}(y_\sharp)$,
\begin{eqnarray*}&&
\int_{B_2'} \widetilde V(y)\,K_*(y,y_\sharp)\,d{\mathcal{H}}^{n-1}_y=
\int_{B_2'} \big(\widetilde V(y)-\widetilde V(y_\sharp)\big)\,K_*(y,y_\sharp)\,d{\mathcal{H}}^{n-1}_y\\
&&\qquad=\int_{B_2'} \big(V(y)-V(y_\sharp)\big)\,K_*(y,y_\sharp)\,d{\mathcal{H}}^{n-1}_y
-\int_{B_2'}\big(P(y)-  P(y_\sharp)\big)\,K_*(y,y_\sharp)\,d{\mathcal{H}}^{n-1}_y\\&&\qquad\le\int_{B_2'} \big(V(y)-V(y_\sharp)\big)\,K_*(y,y_\sharp)\,d{\mathcal{H}}^{n-1}_y
+C_0\\&&\qquad=\int_{\mathcal{R}} \big(v(x)-v(x_\sharp)\big)\,K(x,x_\sharp)\,d{\mathcal{H}}^{n-1}_x
+C_0,
\end{eqnarray*}
with~$C_0$ bounded in terms of~$n$, $s$,
the regularity parameters of~${\mathcal{R}}$ and the structural constant~$C$ of the kernel~$K$
in~\eqref{ojfe-32659867}.

{F}rom this observation and~\eqref{Nkpa-ipjfw40jjo9-RFmra}, we infer that
\begin{eqnarray*}
\int_{B_2'} \widetilde V(y)\,K_*(y,y_\sharp)\,d{\mathcal{H}}^{n-1}_y \le M_0-
a(x_\sharp)v(x_\sharp) +C_0\le 
M_0+\|a_-\|_{L^\infty(\R^n)}+C_0\le
C_0,
\end{eqnarray*}
with~$C_0$ variable from line to line and now depending also on~$M_0$ and~$\|a_-\|_{L^\infty(\R^n)}$.

Since~$\widetilde V\ge0$, we thereby conclude that
\begin{eqnarray*}
\int_{B_{1/100}'(y_\sharp)} \widetilde V(y)\,K_*(y,y_\sharp)\,d{\mathcal{H}}^{n-1}_y\le C_0.
\end{eqnarray*}
We point out that, if~$y\in B_{1/100}'(y_\sharp)$,
\begin{eqnarray*}
K_*(y,y_\sharp)\ge \frac1{C|y-y_\sharp|^{n+s}}\ge \frac{100^{n+s}}{C},
\end{eqnarray*}
and thus
$$\int_{B_{1/100}'(y_\sharp)} \widetilde V(y)\, d{\mathcal{H}}^{n-1}_y
\le \frac{C}{100^{n+s}}
\int_{B_{1/100}'(y_\sharp)} \widetilde V(y)\,K_*(y,y_\sharp)\,d{\mathcal{H}}^{n-1}_y\le C_0.
$$
As a result, up to keeping renaming~$C_0$,
\begin{equation}\label{A7-ipt34g0jf-ew9t4ut9-u43b}\begin{split}&
\int_{B_{1/100}'(y_\sharp)} V(y)\, d{\mathcal{H}}^{n-1}_y\le 
\int_{B_{1/100}'(y_\sharp)} P(y)\,d{\mathcal{H}}^{n-1}_y+ C_0 \le C_0,
\end{split}\end{equation}
where we have also used the fact that~$P(y)\le V(y_\sharp) +M|y_\sharp-y_\star|^2\le1$.

Furthermore,
\begin{eqnarray*}&&
\int_{B_{1/100}'(y_\sharp)}\big( V(y_\sharp)-V(y)\big)\,K_*(y,y_\sharp)\,d{\mathcal{H}}^{n-1}_y\\&&\qquad\le
\int_{B_{1/100}'(y_\sharp)}\big( P(y_\sharp)-P(y)\big)\,K_*(y,y_\sharp)\,d{\mathcal{H}}^{n-1}_y\le C_0
\end{eqnarray*}and therefore one deduces from~\eqref{Nkpa-ipjfw40jjo9-RFmra}
that \begin{eqnarray*}&&
C_0+\int_{{B'_2\setminus B_{1/100}'(y_\sharp)}}
\big( V(y_\sharp)-V(y)\big)\,K_*(y,y_\sharp)
\,d{\mathcal{H}}^{n-1}_y
\\&&\qquad\ge
\int_{B_2'}
\big( V(y_\sharp)-V(y)\big)\,K_*(y,y_\sharp)
\,d{\mathcal{H}}^{n-1}_y
\\&&\qquad=
\int_{\mathcal{R}}
\big( v(x_\sharp)-v(x)\big)\,K(x,x_\sharp)
\,d{\mathcal{H}}^{n-1}_x
\\&&\qquad\ge a(x_\sharp)v(x_\sharp)
-M_0\\&&\qquad\ge -\|a_-\|_{L^\infty(\R^n)}-M_0
\\&&\qquad\ge -C_0
.\end{eqnarray*}
This gives that
\begin{eqnarray*}
&& \int_{B_2'\setminus B_{1/100}'(y_\sharp)}
V(y)\,K_*(y,y_\sharp)\,d{\mathcal{H}}^{n-1}_y\\&&\qquad\le
V(y_\sharp)\int_{B_2'\setminus B_{1/100}'(y_\sharp)}\,K_*(y,y_\sharp)
\,d{\mathcal{H}}^{n-1}_y+C_0\\&&\qquad= C_0\big(V(y_\sharp)+1\big)\\&&\qquad\le C_0
\big(V(y_\star)-M|y_\sharp-y_\star|^2+1\big)\\&&\qquad\le C_0,
\end{eqnarray*}
up to keeping renaming~$C_0$,
and consequently
\begin{equation*}
\int_{B_2'\setminus B_{1/100}'(y_\sharp)}
V(y)\,d{\mathcal{H}}^{n-1}_y\le C_0.\end{equation*}
This and~\eqref{A7-ipt34g0jf-ew9t4ut9-u43b} yield the desired result.
\end{proof}

We now present a useful variant\footnote{Actually, not only the proofs of Lemmata~\ref{ZX-ipfe-jJSL}
and~\ref{ZX-ipfe-jJSL-2-2} are similar, but one could state just one single, albeit more complicated,
result to condensate Lemmata~\ref{ZX-ipfe-jJSL}
and~\ref{ZX-ipfe-jJSL-2-2} into a single statement. For the sake of simplicity, however, we prefer to keep the two statements separate.} of Lemma~\ref{ZX-ipfe-jJSL}.
Its utility in our context is that the geometric equation that we deal with
will present an integral term coming from the nonlocal mean curvature of
two minimal sets that we cannot reabsorb into ``smooth'' objects
(due to the fact that the domain of integration is a ``bad set'', say~${\mathcal{B}}$, containing far away points and points
close to the singular set, for which any regularity information is missing). The next result however will provide
a uniform control of such additional term.

\begin{lemma}\label{ZX-ipfe-jJSL-2-2}
Let~${\mathcal{R}}$ be a portion of a $C^2$ hypersurface
in~$B_2$ that is $C^2$-diffeomorphic to~$B_2\cap\{x_n=0\}$.
Let~$\nu$ be the unit normal vector field of~${\mathcal{R}}$.

Let~$K:{\mathcal{R}}\times{\mathcal{R}}\to[0,+\infty]$ be such that
\begin{equation}\label{ojfe-32659867-mak}\begin{split}&
{\mbox{for all~$x\ne x_0$, we have that }}\;
\frac{1}{C|x-x_0|^{n+s}} \leq K(x,x_0)\le \frac{C}{|x-x_0|^{n+s}},\\ &{\mbox{and, for all~$y\in B_\delta\setminus\{0\}$, }}
|K(y_{x_0}^+,x_0)-K(y_{x_0}^-,x_0)|\le C|y|^{1-n-s},
\end{split}\end{equation}
for some (small) $\delta\in(0,1)$ and some~$C\ge1$.

Assume that, for all~$\widetilde x\in{{\mathcal{R}}}\cap B_{3/2}$, a nonnegative function~$w$ satisfies
\begin{equation}\label{9iuyhgvcgyu8uygfcfy7ytgfP-0oigj-2-223t}\begin{split}&
\int_{\mathcal{R}}
\big( w(\widetilde x)-w(x)\big)\,K(x,\widetilde x)
\,d{\mathcal{H}}^{n-1}_x
-\int_{{\mathcal{B}}}\frac{\Phi(X)}{|X-\widetilde X|^{n+s}}\,dX\ge-\mu
,\end{split}\end{equation}
where~${\mathcal{B}}$ is some measurable set of~$\R^n\setminus{\mathcal{R}}$, $\Phi\in L^\infty(\R^n,[0,1])$,
$\mu\ge0$, and~$\widetilde X:=\widetilde x+w(\widetilde x)\nu(\widetilde x)$.

Assume that there exists a point~$x_0$ in~${\mathcal{R}}\cap B_1$ whose distance from~${\mathcal{B}}$ is bounded from below by some~$r_0>\widetilde C\|w\|_{L^\infty({{\mathcal{R}}}\cap B_2)}$, with~$\widetilde C>1$,
depending only on
the regularity parameters of~${\mathcal{R}}$.

Then, there exists~$C_0>0$, depending only on~$n$, $s$, $r_0$,
the regularity parameters of~${\mathcal{R}}$ and the structural constant~$C$ of the kernel
in~\eqref{ojfe-32659867-mak}, such that
\begin{equation}\label{TBGSHJDK-23eortkg}
\int_{\mathcal{B}}\frac{\Phi(X)}{|X-X_0|^{n+s}}\,dX\le C_0\,\big(
\|w\|_{L^\infty({{\mathcal{R}}}\cap B_2)} +\mu\big),\end{equation}
where~$X_0:=x_0+w(x_0)\nu(x_0)$.
\end{lemma}

\begin{proof} We straighten~${\mathcal{R}}$ by a diffeomorphism~$\phi:{\mathcal{R}}\to B'_2:=B_2\cap\{x_n=0\}$
and set~$W(y):=w(\phi^{-1}(y))$.
We can also suppose for simplicity that~${\mathcal{R}}\cap B_1$ is mapped by~$\phi$ in~$B_1'$ and
that~${\mathcal{R}}\cap B_{3/2}$ is mapped by~$\phi$ in~$B_{3/2}'$.
Let~$x_0\in {\mathcal{R}}\cap B_1$ with~$B_{r_0}(x_0)\cap{\mathcal{B}}=\varnothing$. 
Let also~$y_0:=\phi(x_0)$ and notice that~$y_0\in B_{1}'$.

Given~$M>0$, we slide the parabola~$P(y):= -M|y-y_0|^2$ from below till we touch~$W$ at some point~$\widehat y$. 
The touching condition between the slid parabola and the graph of~$W$ entails that
\begin{equation}\label{oqjwdn234} W(\widehat y)+M|\widehat y-y_0|^2-M|y-y_0|^2\le W(y).\end{equation}
Here we choose
\begin{equation}\label{ipdqjwl0-MSODSnbik00lsd}
M:=C_*^2\,\|w\|_{L^\infty({{\mathcal{R}}}\cap B_2)} \max\left\{1,\frac1{r_0^2}\right\},
\end{equation}with~$C_*$ to be chosen sufficiently large.
In this way, by evaluating~\eqref{oqjwdn234} at~$y:=y_0$,
\begin{eqnarray*}&&
|\widehat y-y_0|\le
\sqrt{\frac{W(y_0)-W(\widehat y)}M}\le \sqrt{\frac{W(y_0)}M}=\sqrt{\frac{w(x_0)}M}\\&&\qquad\le \sqrt{\frac{1}{C_*^2\max\left\{1,\frac1{r_0^2}\right\}}}=\frac1{C_*}\min\{1,r_0\}.
\end{eqnarray*}
In particular, if~$C_*$ is large enough, we have that~$\widehat y\in B_{3/2}'$ and~$|\widehat y-y_0|<\frac{r_0}{10}$.

In this way, the point~$\widehat x:=\phi^{-1}(\widehat y)$
belongs to~${\mathcal{R}}\cap B_{3/2}$.
We can therefore utilize~\eqref{9iuyhgvcgyu8uygfcfy7ytgfP-0oigj-2-223t} at~$\widehat x$, finding that
\begin{equation}\label{i0powjfle2452tg}\begin{split}&\int_{\mathcal{B}}\frac{\Phi(X)}{|X-\widehat X|^{n+s}}\,dX
\le\int_{\mathcal{R}}
\big( w(\widehat x)-w(x)\big)\,K(x,\widehat x)
\,d{\mathcal{H}}^{n-1}_x+\mu\\&\qquad=
\int_{B_2'}
\big( W(\widehat y)-W(y)\big)\,K_*(y,\widehat y)
\,d{\mathcal{H}}^{n-1}_y+\mu
,\end{split}\end{equation}
where~$\widehat X:=\widehat x+w(\widehat x)\nu(\widehat x)$
and we used the notation in~\eqref{02ojfe22-2}.

Notice that, on the one hand, if~$L$ is the Lipschitz constant of~$\phi^{-1}$,
\begin{eqnarray*}&&|X-\widehat X|\le |X-X_0|+|X_0-\widehat X|\le
|X-X_0|+|x_0-\widehat x|+2\|w\|_{L^\infty({{\mathcal{R}}}\cap B_2)} \\&&\qquad\qquad\qquad
\le|X-X_0|+\frac{Lr_0}{10}+2\|w\|_{L^\infty({{\mathcal{R}}}\cap B_2)}\le|X-X_0|+\frac{3Lr_0}{10}.\end{eqnarray*}
On the other hand, if~$X\in{\mathcal{B}}$,
$$ |X-X_0|\ge |X-x_0|-\|w\|_{L^\infty({{\mathcal{R}}}\cap B_2)}\ge r_0-\|w\|_{L^\infty({{\mathcal{R}}}\cap B_2)}\ge\frac{9r_0}{10}.$$
By combining these observations, we find that
$$ |X-\widehat X|\le |X-X_0|+\frac{3L}{10}\cdot\frac{10}9|X-X_0|\le C_0|X-X_0|,$$for some~$C_0>1$.

This and~\eqref{i0powjfle2452tg}, up to renaming~$C_0$, give that
\begin{equation*}\begin{split}&\frac1{C_0}\int_{\mathcal{B}}\frac{\Phi(X)}{|X-X_0|^{n+s}}\,dX
\le\int_{B_2'}
\big( W(\widehat y)-W(y)\big)\,K_*(y,\widehat y)
\,d{\mathcal{H}}^{n-1}_y+\mu.\end{split}\end{equation*}
As a consequence, by~\eqref{oqjwdn234},
\begin{equation*}\begin{split}\frac1{C_0}\int_{\mathcal{B}}\frac{\Phi(X)}{|X-X_0|^{n+s}}\,dX
\le\;& M\int_{B_2'}\big( |y-y_0|^2-|\widehat y-y_0|^2\big)\,K_*(y,\widehat y)
\,d{\mathcal{H}}^{n-1}_y+\mu\\ \le\;& C_0\big(M+\mu\big).\end{split}\end{equation*}
The desired result thus follows in view of~\eqref{ipdqjwl0-MSODSnbik00lsd}.
\end{proof}

\section{Completion of the proof of Theorem~\ref{MAIN}}\label{HAR:NA}

To complete the proof of Theorem~\ref{MAIN}, we will make use of some technical results about the set
of ``good'' regular points. Namely, we describe the set of  regular points of a nonlocal minimal set
around which the boundary is a hypersurface of class~$C^2$ with bounded curvatures in a conveniently scale invariant setting. In the classical case, this set was introduced in~\cite[Lemma~1]{MR906394} and we provide here a nonlocal version of it.

\begin{lemma}\label{joqslwd-owjefg838r-1}
Let~$E$ be $s$-minimal in~$B_1$ and assume that
the origin belongs to its boundary.

Given~$\theta>0$, let~${\mathcal{M}}_{\theta,E}$
be the set of points~$x$
belonging to the regular part of~$\partial E$
such that
\begin{equation}\label{ojDSDV0RTGPRHO}\begin{split}&
{\mbox{$(\partial E)\cap B_{\theta|x|}(x)$ is contained in the regular part of~$\partial E$}}\\&
{\mbox{and possesses curvatures bounded in absolute value by }}\frac1{\theta|x|}.
\end{split}\end{equation}

Then, there exist~$\rho_0$, $\theta_0>0$, depending only on~$E$, $n$ and~$s$, such that, for every~$\rho\in(0,\rho_0]$ and~$\theta\in(0,\theta_0]$, we have that
\begin{equation}\label{12o0asoiIACCU} {\mathcal{M}}_{\theta,E}\cap(\partial B_\rho)\ne\varnothing.\end{equation}
\end{lemma}

With reference to Lemma~\ref{joqslwd-owjefg838r-1}, we observe that, for all~$r>0$,
\begin{equation}\label{oqjdlf-12poel}
{\mathcal{M}}_{\theta,\frac{E}r}=\frac{{\mathcal{M}}_{\theta,E}}r.
\end{equation}
Then, we have:

\begin{corollary}\label{12o0asoiIACCU12}
Let~$E_1$, $E_2$ be $s$-minimal sets in~$B_1$ and assume that
the origin belongs to their boundary.
Assume also that~$E_1$ and~$E_2$ have the same tangent cone at the origin.

Let~$\theta_0$, $\rho_0>0$ be as in Lemma~\ref{joqslwd-owjefg838r-1},
used here with~$E:=E_1$.
Let also~${\mathcal{M}}$ be the set~${\mathcal{M}}_{\theta_0,E}$
in Lemma~\ref{joqslwd-owjefg838r-1} with~$E=E_1$.

Suppose that
$$ E_2\setminus E_1=\Big\{
x+t\nu(x),\;{\mbox{ with }}\; x\in{\mathcal{M}},\; t\in\big[0,w(x)\big)
\Big\},$$
for some function~$w$ of class~$C^{2}$, where~$\nu$ denotes the external unit normal of~$E_1$ at its regular points.

Then,
\begin{equation}\label{12o0asoiIACCU1} \lim_{r\searrow0}
\frac{\displaystyle\sup_{x\in(\partial B_r)\cap{\mathcal{M}}}|w(x)|}r=0.\end{equation}

In addition, if~$w(x)\ne0$ for all~$x\ne0$, then there exists an infinitesimal sequence of points~$z_{\star,k} \in {\mathcal{M}}$ such that
\begin{equation}\label{12o0asoiIACCU13} \frac{|w(x)|}{|x|} \le \frac{2 |w(z_{\star,k})|}{|z_{\star,k}|}\end{equation}
 for all~$x\in {\mathcal{M}}$ with~$|x| \le |z_{\star,k}|/2$.
 
Finally, the distance of~$z_{\star,k}$ from the singular set of~$E_1$ is at least~$\theta_0|z_{\star,k}|$.
\end{corollary}

For the convenience of the reader, the proofs of Lemma~\ref{joqslwd-owjefg838r-1}
and Corollary~\ref{12o0asoiIACCU12}
are postponed to Appendix~\ref{GOODAP}.
\medskip

We can now complete the proof of Theorem~\ref{MAIN}  by relying on the work carried out so far and on the Harnack Inequality by Cabr\'e and Cozzi in~\cite{MR3934589}.

\begin{proof}[Proof of Theorem~\ref{MAIN}]
Suppose, by contradiction, that~$E_1\ne E_2$.
Without loss of generality, we may suppose that
\begin{equation}\label{DESM}
{\mbox{$E_1$ and~$E_2$ share the same tangent cone at the origin.}}
\end{equation}
This is a standard procedure in geometric measure theory, based on blow-up, dimensional reduction, and regularity theory. We recall the details for the facility of the reader.

Indeed, if~\eqref{DESM} does not hold, we have a converging blow-up sequence for~$E_1$ approaching a cone~${\mathcal{C}}_1$
with the corresponding blow-up sequence for~$E_2$ approaching a cone~${\mathcal{C}}_2\ne{\mathcal{C}}_1$.
We take a rotation~${\mathcal{R}}_1$ of~${\mathcal{C}}_1$ such that~$E_1^{(1)}:=
{\mathcal{R}}_1{\mathcal{C}}_1\subseteq {\mathcal{C}}_2=:E_2^{(1)}$ and there exists~$p^{(1)}\in(\partial E_1^{(1)})\cap (\partial E_2^{(1)})\cap (\partial B_1)$.

We then consider a blow-up of~$E_1^{(1)}$ and~$E_2^{(1)}$ at~$p^{(1)}$,
which produces two new $s$-minimal cones in~$\R^n$, which will be denoted by~${\mathcal{C}}_1^{(1)}$
and~$ {\mathcal{C}}_2^{(2)}$, with~${\mathcal{C}}_1^{(1)}
\subseteq{\mathcal{C}}_2^{(2)}$.

Now, two cases can occur. If~${\mathcal{C}}_1^{(1)}={\mathcal{C}}_2^{(2)}$, it suffices to replace~$p$,
$E_1$, $E_2$, ${\mathcal{C}}_1$ and~${\mathcal{C}}_2$ respectively with~$p^{(1)}$,
$E_1^{(1)}$, $E_2^{(1)}$, ${\mathcal{C}}_1^{(1)}$ and~${\mathcal{C}}_2^{(1)}$: in this way 
we have obtained~\eqref{DESM} in this new configuration.

If instead~${\mathcal{C}}_1^{(1)}\ne{\mathcal{C}}_2^{(1)}$, we observe that
both these cones are cylinders over~$\R^{n-1}$. In this way, by the dimensional reduction (see~\cite{MR2675483}),
we have found two $s$-minimal cones in~$\R^{n-1}$, which we denote by~$\widetilde E_1^{(2)}$ and~$\widetilde E_2^{(2)}$,
such that~$\widetilde E_1^{(2)}\subseteq \widetilde E_2^{(2)}$ and~$\widetilde E_1^{(2)}\ne \widetilde E_2^{(2)}$ (and, up to an isometry, ${\mathcal{C}}_j^{(1)}=\widetilde E_j^{(2)}\times\R$ for~$j\in\{1,2\}$).

We thus repeat the previous algorithm, taking a rotation~${\mathcal{R}}_2$ in~$\R^{n-1}$, such that~$E_1^{(2)}:=
{\mathcal{R}}_2\widetilde E_1^{(2)}\subseteq \widetilde{E}_2^{(1)}=:E_2^{(1)}$ and there exists~$p^{(2)}\in(\partial E_1^{(2)})\cap (\partial E_2^{(2)})\cap (\partial B_1)$. Then, a blow-up at~$p^{(2)}$ produces
two $s$-minimal cones~${\mathcal{C}}_1^{(2)}$ and~${\mathcal{C}}_2^{(2)}$, which either coincide
(whence we replace~$p$,
$E_1$, $E_2$, ${\mathcal{C}}_1$ and~${\mathcal{C}}_2$ respectively with~$p^{(2)}$,
$E_1^{(2)}$, $E_2^{(2)}$, ${\mathcal{C}}_1^{(2)}$ and~${\mathcal{C}}_2^{(2)}$) or not, in which case we apply again the dimensional reduction (reducing now to~$\R^{n-2}$) and proceed.

This algorithm will stop at a certain dimension, producing the same $s$-minimal cones after a blow-up, since, by the regularity of $s$-minimal cones in dimension~$2$
(see~\cite{MR3090533}), we know that halfspaces are the only nontrivial minimal cones in~$\R^2$.
The proof of~\eqref{DESM} is thereby complete.

Hence, from now on, we will denote by~${\mathcal{C}}$
the common tangent cone of~$E_1$ and~$E_2$ along a blow-up sequence that we are now going to specify.
To this end, in a neighborhood of the origin,
we write~$E_2$ in normal coordinates with respect to~$E_1$ at its regular points, that is suppose that, at a set of regular points,
$E_2\setminus E_1$ has the form~$x+t\nu(x)$, with~$x\in\partial E_1$ and~$t\in\big[0,w_0(x)\big)$,
for a suitable function~$w_0>0$.
Then, by Corollary~\ref{12o0asoiIACCU12}, we find a set~${\mathcal{M}}_0$ of regular points for~$E_1$
and an infinitesimal sequence of points~$z_{\star,k} \in {\mathcal{M}}_0$ such that
\begin{equation} \label{Aappcjspd-01}\frac{w_0(x)}{|x|} \le \frac{2 w_0(z_{\star,k})}{|z_{\star,k}|}\end{equation}
for all~$x\in {\mathcal{M}}_0$ with~$|x| \le |z_{\star,k}|/2$.
In addition, by~\eqref{ojDSDV0RTGPRHO},
\begin{equation}\label{0DIFUJH-0wiuoe2irhfegrf-iqpdwojfekgu0rjog}
{\mbox{the distance of any point~$x\in{\mathcal{M}}_0$ from the singular set is at least~$\theta_0|x|$,}}\end{equation}
with~$\theta_0>0$.

What is more, by~\eqref{12o0asoiIACCU}, we also know that, for all~$\rho\in(0,\rho_0]$,
\begin{equation}\label{0DIFUJH-0wiuoe2irhfegrf-iqpdwojfekgu0rjog2}
{\mathcal{M}}_0\cap(\partial B_{\rho})\ne\varnothing.\end{equation}
Hence, we choose
\begin{equation} \label{Aappcjspd-02}
r_k:=|z_{\star,k}|\end{equation} and consider the corresponding blow-up sequences for~$j\in\{1,2\}$ defined by 
 \begin{equation}\label{EJK}
E_{j,k}:=\frac{E_j}{r_k}.\end{equation}

We now apply Theorem~\ref{PRIMOLE-TH}. More specifically, the sets~$E_1$ and~$E_2$ in the statement of Theorem~\ref{PRIMOLE-TH}
are here the sets~$E_{1,k}$ and~$E_{2,k}$, as defined in~\eqref{EJK}, with~$k$ large enough.

The gist of the construction is that we cover the singular set of~${\mathcal{C}}$ by
a set of small balls (the fact that the singular set has dimension at most~$n-3$,
due to~\cite{MR3090533, MR3331523} guarantees that these balls can be chosen to satisfy~\eqref{AKSMHAU} for~$\eta$ as small as we wish).

Since, for~$k$ large enough,~$E_{1,k}$ and~$E_{2,k}$ locally lie in a small neighborhood of~${\mathcal{C}}$, we have that both~$E_{1,k}$ and~$E_{2,k}$ are smooth away from the above covering of the singular set of~${\mathcal{C}}$, thanks to the improvement of flatness of nonlocal minimal surfaces put forth in~\cite[Corollary~4.4 and Theorem~6.1]{MR2675483}
(in this way, the convergence of~${E_{1,k}}$ and~${E_{2,k}}$ to the limit cone~${\mathcal{C}}$ occurs in the~$C^{1,\alpha}$ sense away from any small neighborhood of the singular set of~${\mathcal{C}}$, and actually in the~$C^k$ sense for any~$k\ge2$, thanks to the bootstrap regularity in~\cite{MR3331523}).

This allows us to look at a ``large'' set\footnote{Note that this set~${\mathcal{G}}$ possibly contains~${\mathcal{M}}_0$.} 
${\mathcal{G}}$ containing regular points as in Theorem~\ref{PRIMOLE-TH}
(and at a small shrinkage of it, namely~${\mathcal{G}}'$) such that all the singular points
of~$\partial{\mathcal{C}}$, $\partial E_1$ and~$\partial E_2$ in a large ball~$B_R$
lie outside~${\mathcal{G}}$ (also, ${\mathcal{G}}$ covers all~$B_R$, up to a negligible covering of balls, as specified in~\eqref{AKSMHAU}).

Accordingly, for~$j\in\{1,2\}$ and~$k$ sufficiently large,
one can parameterize~${E_{j,k}}$ as a graph of class~$C^2$ in the normal direction of~${\mathcal{C}}$
away from a small neighborhood of the singular set. 

It is however technically simpler to recenter this parameterization on the ``middle surface'', namely to parameterize~$E_2=E_{2,k}$ in terms of~$E_1=E_{1,k}$, and this corresponds to the
normal parameterization~$w=w_k$ in Theorem~\ref{PRIMOLE-TH}
(here, following~\eqref{EJK}, we are taking~$w_k(x):=\frac{w_0(r_kx)}{r_k}$ for~$x\in(\partial E_{1,k})\cap{\mathcal{G}}$).

We stress that since both~$E_{1,k}$ and~$E_{2,k}$ converge to~${\mathcal{C}}$ locally in the~$C^2$-sense away from the singular points of~$\partial{\mathcal{C}}$, we also know
that the $C^2$-norm of~$w=w_k$ is as small as we like, provided that~$k$ is chosen large enough (possibly in dependence also of the covering of the singular set, which is now fixed once and for all).

We can therefore utilize Theorem~\ref{PRIMOLE-TH} in this context and conclude that the normal parameterization~$w=w_k$ of~$E_2=E_{2,k}$ with respect to~$E_1=E_{1,k}$
satisfies the equation in~\eqref{o9eudashiqpowudjhifeg}, with suitable integrable kernels, as defined in~\eqref{IPSJLkear34}. More specifically, since~$E_1$ and~$E_2$ are $s$-minimal sets and therefore~$H^s_{E_1}=H^s_{E_2}$ at every regular boundary point
(see~\cite[Theorem~5.1]{MR2675483}), we obtain from Theorem~\ref{PRIMOLE-TH} that
\begin{equation}\label{9iuyhgvcgyu8uygfcfy7ytgfP-0oigj}\begin{split}&0
=\int_{(\partial E_1)\cap {\mathcal{G}}'}
\big( w(x_0)-w(x)\big)\,K_1(x,x_0)
\,d{\mathcal{H}}^{n-1}_x\\&\qquad\qquad-w(x_0)\int_{(\partial E_1)\cap {\mathcal{G}}'}
\big(1-\nu(x_0)\cdot\nu(x)\big)\,K_2(x,x_0)
\,d{\mathcal{H}}^{n-1}_x\\&\qquad\qquad
+\frac12\int_{\R^n\setminus {\mathcal{G}}'}\frac{\widetilde\chi_{E_2}(X)-\widetilde\chi_{E_1}(X)}{|X-X_0|^{n+s}}\,dX
+O\left( w(x_0)\eta\right)
+O\left(\frac{w(x_0)}{R^{1+s}}\right),\end{split}\end{equation}
where~$\nu=\nu_k$ is the outer unit normal of~$\partial E_1=\partial E_{1,k}$ at its regular points and~$X_0:=x_0+w(x_0)\nu(x_0)$.

Strictly speaking, the function~$w$ does not need to be defined everywhere in~$\R^n$ for the validity of~\eqref{9iuyhgvcgyu8uygfcfy7ytgfP-0oigj}, since such an equation does not consider values of~$w$ outside certain sets, but we will implicitly suppose that~$w$ is defined everywhere for definiteness.

We now take~${\mathcal{G}}''\Subset{\mathcal{G}}'$ and~$x_0\in(\partial E_1)\cap B_{R/2}\cap{\mathcal{G}}''$
(as mentioned on page~\pageref{aojscnSMNEdcvSD3jx},
this gives that~${\mathcal{G}}''$ lies in a small neighborhood of~$\partial E_1$).

In this setting, owing to Lemma~\ref{ZX-ipfe-jJSL-2-2} and to the fact that~$E_1\subseteq E_2$, we see that, when the size of~$w$ is small,
\begin{equation}\label{TBGSHJDK-23eortkg-2-2}0\le
\int_{\R^n\setminus {\mathcal{G}}'}\frac{\widetilde\chi_{E_1}(X)-\widetilde\chi_{E_2}(X)}{|X-X_0|^{n+s}}\,dX= O(\|w\|_{L^\infty(B_R\cap{\mathcal{G}})}).
\end{equation}

Now, since~$E_{1}\not=E_{2}$, we have that~$w$ does not vanish identically, and in fact~$w>0$ on the regular part of~$\partial E_1$
(otherwise, we could compute the nonlocal mean curvature of~$E_1$ pointwise using the smoothness of~$\partial E_1$, as well as the one of~$E_2$
in the viscosity sense and get a contradiction from the two integral contributions and the set inclusion).
Therefore we can normalize~$w$ to take value~$1$ at a suitable point.
Roughly speaking, it would be desirable to pick this point to reach the supremum of~$\frac{ w(x)}{|x|}$, so to obtain a ``linear'' separation about the minimal sheets at the origin and thus contradict~\eqref{DESM}: however, this choice
of the normalizing point may be impossible, since the supremum of~$\frac{ w(x)}{|x|}$ might well be on the singular set (where~$w$ is not even defined)
or dangerously close to it.
This is the reason for which we have chosen an appropriate sequence of blow-up radii~$r_k$ in~\eqref{Aappcjspd-01}
and~\eqref{Aappcjspd-02}.

In this setting, we define~$x_{\star,k}:=\frac{z_{\star,k}}{|z_{\star,k}|}=\frac{z_{\star,k}}{r_k}$
and, in light of~\eqref{0DIFUJH-0wiuoe2irhfegrf-iqpdwojfekgu0rjog},
we stress that the distance of~$x_{\star,k}$ from the singular set of~$\partial E_{1,k}$ is at least~$\theta_0$.
This ensures that, up to a subsequence, $x_{\star,k}$ converges\footnote{Let us stress that the fact that the limit point~$x_{\star,\infty}$ remains bounded and
bounded away from the singular set is essential to apply the
Harnack Inequality in~\cite{MR3934589}: this is the reason for which
we can well allow intermediate constants to depend
on the radius of the large
ball that we are considering, on the set of small balls chosen to
cover the singular set, as well as on the regularity of the $s$-minimal sheets
away of this cover, since we will pass~$k\to+\infty$
before removing this large ball and this small cover,
but we will employ the normalization at the limit point~$x_{\star,\infty}$
to bound from below the last term in~\eqref{piwdjfk02-24}.}
to a regular point~$x_{\star,\infty}$ of the limit cone as~$k\to+\infty$.

Furthermore, by~\eqref{0DIFUJH-0wiuoe2irhfegrf-iqpdwojfekgu0rjog2}, given~$\delta\in\left(0,\frac12\right]$,
to be chosen conveniently small in what follows,
we can pick a point~$y_k\in
{\mathcal{M}}_0\cap(\partial B_{\delta r_k})$. In this way, by~\eqref{0DIFUJH-0wiuoe2irhfegrf-iqpdwojfekgu0rjog},
we also have that the distance of~$y_{\star,k}:=\frac{y_k}{r_k}$ from the singular set of~$\partial E_{1,k}$ is at least~$\delta\theta_0$
and therefore, up to a subsequence, $y_{\star,k}$ converges to a regular point~$y_{\star,\infty}$ of the limit cone as~$k\to+\infty$.

We also recall~\eqref{Aappcjspd-01} and write that
\begin{eqnarray*}
\frac{w_0(y_k)}{\delta r_k}=\frac{w_0(y_k)}{|y_k|} \le \frac{2 w_0(z_{\star,k})}{|z_{\star,k}|}=
\frac{2 w_0(z_{\star,k})}{r_k}
\end{eqnarray*}
and accordingly
\begin{equation}\label{TBcosdkcmTRAnsdc0o2ekdmf02erfgvb}
w_k(y_{\star,k})\le
2 \delta w_k(x_{\star,k}).
\end{equation}

We now define
$$ v(x)=v_k(x):=\frac{w_k(x)}{w(x_{\star,k})}$$ and we divide~\eqref{9iuyhgvcgyu8uygfcfy7ytgfP-0oigj} by~$w(x_{\star,k})=w_k(x_{\star,k})$, recalling also~\eqref{TBGSHJDK-23eortkg-2-2}, to find that, for all~$x_0\in(\partial E_1)\cap B_{R/2}\cap{\mathcal{G}}''$,
\begin{equation}\label{iposjakdhfdVASNZo}\begin{split}&0
=\int_{(\partial E_1)\cap {\mathcal{G}}'}
\big( v(x_0)-v(x)\big)\,K_1(x,x_0)
\,d{\mathcal{H}}^{n-1}_x\\&\qquad\qquad-v(x_0)\int_{(\partial E_1)\cap {\mathcal{G}}'}
\big(1-\nu(x_0)\cdot\nu(x)\big)\,K_2(x,x_0)
\,d{\mathcal{H}}^{n-1}_x\\&\qquad\qquad
-\psi(x_0)
+O\left(\eta\right)
+O\left(\frac{1}{R^{1+s}}\right),\end{split}\end{equation}
where~$0\le\psi=O(1)$.

We recall that we are using the short notation for the set~$E_1=E_{1,k}$.

Also, if~${\mathcal{G}}'''\Subset{\mathcal{G}}''$, we have that, in~$(\partial E_1)\cap B_{R/4}\cap{\mathcal{G}}'''$,
\begin{equation}\label{DDJH-SKDrtfg}
0\le v\le O(1).\end{equation} 
Indeed, we first observe that~$v\ge0$, since~$w\ge0$. The upper bound in~\eqref{DDJH-SKDrtfg}
follows from the previously developed regularity theory.
Specifically, one first employs Lemma~\ref{ZX-ipfe-jJSL}
and finds that
\begin{equation*}
\int_{(\partial E_1)\cap {\mathcal{G}}'}
v(x)\,d{\mathcal{H}}^{n-1}_x\le O(1).\end{equation*}
This inequality provides the uniform bound corresponding to~\eqref{XM-skdadM-BIS} in the present setting.
Now, the upper bound in~\eqref{DDJH-SKDrtfg} is a consequence of Lemma~\ref{VOL}. The proof of~\eqref{DDJH-SKDrtfg} is thereby complete.

Accordingly, by the H\"older estimates in Lemma~\ref{DERGED-Ge}, we deduce that the $C^{\alpha}$-norm of~$v$ is bounded locally uniformly in~$k$ (and we stress that, since~$w=w_k$, also~$v=v_k$).

Therefore, up to a subsequence, $v_k$ converges locally uniformly in~$B_{R/4}\cap{\mathcal{G}}'''$ to a function~$v_\infty$. In view of~\eqref{iposjakdhfdVASNZo},
$v_\infty$ satisfies, for all~$x_0\in (\partial E_1)\cap B_{R/4}\cap{\mathcal{G}}'''$,
\begin{equation}\label{ojdcn834rifgjvCRoFRr45VIS}\begin{split}&0
\le\int_{(\partial {\mathcal{C}})\cap {\mathcal{G}}'}
\frac{v_\infty(x_0)-v_\infty(x)}{|x-x_0|^{n+s}}
\,d{\mathcal{H}}^{n-1}_x\\&\qquad\qquad-v_\infty(x_0)\int_{(\partial{\mathcal{C}})\cap {\mathcal{G}}'}\frac{1-\nu(x_0)\cdot\nu(x)}{{|x-x_0|^{n+s}}}
\,d{\mathcal{H}}^{n-1}_x
+O\left(\eta\right)
+O\left(\frac{1}{R^{1+s}}\right).\end{split}\end{equation}
We stress that we have used here above the fact that~$\psi(x_0)\ge0$
and also that~\eqref{ojdcn834rifgjvCRoFRr45VIS} is intended in the viscosity sense:
indeed, to obtain~\eqref{ojdcn834rifgjvCRoFRr45VIS} from~\eqref{iposjakdhfdVASNZo}
one can project~$(\partial E_1)\cap {\mathcal{G}}'$ on the underlying cone, and
use the asymptotics of the kernel and the stability of viscosity solutions
under locally uniform convergence to pass to the limit.

Also, we can now take~$R$ in~\eqref{ojdcn834rifgjvCRoFRr45VIS}
as large as we wish and invade all the space outside the singular set by the ``good'' domains~${\mathcal{G}}$, ${\mathcal{G}}'$ and~${\mathcal{G}}''$.
In this way, we find that,
on the regular part~$\Reg{\mathcal{C}}$ of~$\partial{\mathcal{C}}$,
\begin{equation*}\begin{split}&0
\le\int_{\Reg {\mathcal{C}}}
\frac{ v_\infty(x_0)-v_\infty(x)}{|x-x_0|^{n+s}}
\,d{\mathcal{H}}^{n-1}_x -v_\infty(x_0)\int_{\Reg{\mathcal{C}}}
\frac{1-\nu(x_0)\cdot\nu(x)}{|x-x_0|^{n+s}}
\,d{\mathcal{H}}^{n-1}_x.\end{split}\end{equation*}

Now, we wish to apply Proposition~\ref{LANSCDPOJFIBFerngvny83iurjf} to~$v_\infty$
and~${\mathcal{R}}:=\Reg {\mathcal{C}}$. To this end, we observe that~\eqref{o3ePP-1-pre}
holds true, due to the regularity theory of nonlocal minimal surfaces, see~\cite{MR3090533, MR3331523},
and~\eqref{o3ePP-1}
is satisfied, thanks to the perimeter estimates
for nonlocal minimal surfaces, see~\cite[equation~(1.16)]{MR3981295}.

It thus follows
from Proposition~\ref{LANSCDPOJFIBFerngvny83iurjf} that,
for every compact subset~${\mathcal{K}}$ of~$\Reg {\mathcal{C}}$,
\begin{equation}\label{quwje13POvq} \inf_{{\mathcal{K}}\cap B_1} v_\infty
\ge c\left( \int_{{\mathcal{K}}\cap B_1}v_\infty(x)\,d{\mathcal{H}}^{n-1}_x+\int_{{\mathcal{K}}\setminus B_1}\frac{v_\infty(x)}{|x|^{n+s}}\,d{\mathcal{H}}^{n-1}_x\right),\end{equation}
for some~$c>0$, depending only on~$n$ and~$s$.

Moreover,
\begin{equation*}v_\infty(x_{\star,\infty})=\lim_{k\to+\infty}v_k(x_{\star,k})=1.\end{equation*}
Using this information into~\eqref{quwje13POvq}, we thus obtain
\begin{equation} \label{piwdjfk02-24}
m:=\inf_{{\mathcal{K}}\cap B_1} v_\infty\ge c\int_{{\mathcal{K}}\cap B_1}v_\infty(x)\,d{\mathcal{H}}^{n-1}_x>0.\end{equation}
As a result, for every compact subset~${\mathcal{K}}$ of~$\Reg {\mathcal{C}}\cap B_{3/4}$, we have that
$$ \inf_{x\in {\mathcal{K}}}\frac{w_k(x)}{w_k(x_{\star,k})}=\inf_{x\in {\mathcal{K}}}v_k(x)\ge\frac{m}2,$$
as long as~$k$ is large enough (possibly in dependence of~${\mathcal{K}}$).

In particular,
$$ \frac{w_k(y_{\star,k})}{w_k(x_{\star,k})}\ge\frac{m}2,$$
which gives a contradiction with~\eqref{TBcosdkcmTRAnsdc0o2ekdmf02erfgvb} when~$\delta$ is sufficiently small.\end{proof}

\begin{appendix}

\section{The set of ``good'' regular points}\label{GOODAP}

Here we prove Lemma~\ref{joqslwd-owjefg838r-1} and Corollary~\ref{12o0asoiIACCU12}.

\begin{proof}[Proof of Lemma~\ref{joqslwd-owjefg838r-1}] Suppose the claim of Lemma~\ref{joqslwd-owjefg838r-1} does not hold true. Then, there exist sequences~$\theta_j\searrow0$ and~$\rho_j\searrow0$
for which~${\mathcal{M}}_{\theta_j,E}\cap(\partial B_{\rho_j})=\varnothing$.
That is, if~$x\in\partial B_{\rho_j}$
belongs to the regular part of~$\partial E$ and
also~$(\partial E)\cap B_{\theta_j|x|}(x)$ is contained the regular part of~$\partial E$, then necessarily its curvatures are not bounded 
in absolute value by~$\frac1{\theta_j|x|}=\frac{1}{\theta_j\rho_j}$.

Let~$F_j:=E/\rho_j$. Up to a subsequence, we know that~$F_j$ converges locally uniformly to a cone~${\mathcal{C}}$ which is $s$-minimal (see~\cite{MR2675483}). In fact, thanks to the improvement of flatness of nonlocal minimal surfaces (see~\cite[Corollary~4.4 and Theorem~6.1]{MR2675483}), this convergence occurs in the~$C^{1,\alpha}$ sense at the regular points of~$\partial{\mathcal{C}}$, and actually in the~$C^k$ sense for any~$k\ge2$ (see~\cite{MR3331523}).

Hence, we pick a regular point~$y_0\in(\partial{\mathcal{C}})\cap (\partial B_1)$ 
(whose existence is guaranteed by the regularity theory of nonlocal minimal surfaces, see~\cite{MR3090533}). Then, we find a sequence~$y_j\in\partial B_1$
of regular points of~$\partial F_j$ approaching~$y_0$ as~$j\to+\infty$
and~$(\partial F_j)\cap B_{r_0}(y_j)$ consists of regular points
with curvatures bounded in absolute value by~$M_0$, for suitable~$r_0$, $M_0>0$.

Scaling back and setting~$x_j:=\rho_j y_j\in\partial B_{\rho_j}$, we find that~$(\partial E)\cap B_{r_0\rho_j}(x_j)$ lies
in the regular set of~$\partial E$, with
curvatures bounded in absolute value by~$\frac{M_0}{\rho_j}$, which, for large~$j$, is strictly less than~$\frac1{\theta_j\rho_j}$, contradiction.
\end{proof}

\begin{proof}[Proof of Corollary~\ref{12o0asoiIACCU12}]
Before proving~\eqref{12o0asoiIACCU1}, we stress that
the supremum in this equation makes sense, thanks to~\eqref{12o0asoiIACCU}.

We now prove~\eqref{12o0asoiIACCU1}.
Suppose not. Then, there exist~$c>0$, an infinitesimal sequence~$r_j>0$ and~$
x_j\in(\partial B_{r_j})\cap{\mathcal{M}}$
such that
\begin{equation}\label{oksmxcABodkjLmBOaksmx1}  |w(x_j)|\ge c r_j.\end{equation}
Thus, for~$i\in\{1,2\}$, we define~$E_{i,j}:=\frac{E_i}{r_j}$ and, for all~$x\in\frac{{\mathcal{M}}}{r_j}$,
$$ w_j(x):=\frac{w(r_jx)}{r_j}.$$
In this way, we see that
$$ E_{2,j}\setminus E_{1,j} =\left\{
x+t\nu_j(x),\;{\mbox{ with }}\;x\in\frac{{\mathcal{M}}}{r_j},\; t\in\big[0,w_j(x)\big)
\right\},$$
where~$\nu_j$ denotes the external unit normal of~$E_{1,j}$.

Combining this and the notation in~\eqref{oqjdlf-12poel}, we obtain
$$ E_{2,j}\setminus E_{1,j}  =\Big\{
x+t\nu_j(x),\;{\mbox{ with }}\; x\in{\mathcal{M}}_{\theta_0,E_{1,j}},\; t\in\big[0,w_j(x)\big)\Big\}.$$
So, setting~$y_j:=\frac{x_j}{r_j}\in(\partial B_1)\cap {\mathcal{M}}_{\theta_0,E_{1,j}}$, we have that
\begin{equation}\label{oksmxcABodkjLmBOaksmx}
y_j+w_j(y_j)\nu_j(y_j)\in \partial E_{2,j}.\end{equation}

Notice in addition that, since the distance of~$x_j$ from the singular set of~$E_1$
is at least~$\theta_0|x_j|=\theta_0r_j$, we have that the
distance of~$y_j$ from the singular set of~$E_{1,j}$
is at least~$\theta_0$. The curvature bound thus allows us to pass to the limit as~$j\to+\infty$, up to a subsequence, and find that~$y_j\to y_\infty$,
with~$y_\infty$ belonging to the regular part of the tangent cone of~$E_1$,
and also~$\nu_j(y_j)\to\nu_\infty(y_\infty)$, where~$\nu_\infty$ denotes the normal of this tangent cone.

Hence, since~$E_2$ shares the same tangent cone of~$E_1$, it follows from~\eqref{oksmxcABodkjLmBOaksmx} that~$w_j(y_j)\to0$ as~$j\to+\infty$.
But this is in contradiction with~\eqref{oksmxcABodkjLmBOaksmx1}
and the proof of~\eqref{12o0asoiIACCU1} is thereby complete.

Now we prove~\eqref{12o0asoiIACCU13}. For this, let
$$ \sigma(r):=\frac{\displaystyle\sup_{x\in(\partial B_r)\cap{\mathcal{M}}}|w(x)|}r.$$
Thanks to~\eqref{12o0asoiIACCU1}, we have that~$\sigma(r)\to0$
as~$r\searrow0$.

Also, we know that~$\sigma(r)>0$ for all~$r\ne0$.
Therefore, given~$\mu_0>1$ we pick an infinitesimal sequence~$r_k>0$ and choose~$r_{\star,k}\in(0,r_k]$ such that
$$ \sigma(r_{\star,k})\ge \frac{\displaystyle\sup_{r\in(0,r_k]}\sigma(r)}{\mu_0}.$$
As a result, for all~$\mu\in(0,1)$,
\begin{eqnarray*}&&\frac{\displaystyle
\sup_{x\in(\partial B_{\mu{r_{\star,k}}})\cap{\mathcal{M}}}|w(x)|}{\mu {r_{\star,k}}}=
\sigma({\mu{r_{\star,k}}})\le
\sup_{r\in(0,r_k]}\sigma(r)\le
\mu_0\,\sigma(r_{\star,k}).
\end{eqnarray*}

Besides, by~\eqref{12o0asoiIACCU}, we can pick~$z_{\star,k}\in(\partial B_{r_{\star,k}})\cap{\mathcal{M}}$
with
$$ \frac{|w(z_{\star,k})|}{r_{\star,k}}\ge\frac{\displaystyle\sup_{x\in(\partial B_{r_{\star,k}})\cap{\mathcal{M}}}|w(x)|}{\mu_0 \,r_{\star,k}}
=\frac{ \sigma(r_{\star,k})}{\mu_0}.$$
This gives that, for all~$x\in {\mathcal{M}}$ with~$|x| \le|z_{\star,k}|/2$,
we have that
\begin{eqnarray*}
&&  \frac{|w(x)|}{|x|} \le \sup_{\mu\in(0,1/2]}
\frac{\displaystyle\sup_{x\in(\partial B_{\mu{r_{\star,k}}})\cap{\mathcal{M}}}|w(x)|}{\mu r_{\star,k}}\le
\mu_0\,\sigma(r_{\star,k})
\le
\frac{\mu_0^2 \,|w(z_{\star,k})|}{|z_{\star,k}|}
,\end{eqnarray*}
which, choosing~$\mu_0:=\sqrt2$, establishes~\eqref{12o0asoiIACCU13}, as desired.

Finally, the distance of~$z_{\star,k}$ from the singular set is estimated from below by the first line in~\eqref{ojDSDV0RTGPRHO}.
\end{proof}

\end{appendix}

\end{document}